\RequirePackage{fix-cm}
\documentclass{amsart}
\usepackage{graphicx}
\usepackage{mathptmx}
\usepackage{amsmath}
\usepackage{amssymb}
\usepackage{amsthm}
\usepackage{calrsfs}
\usepackage{enumerate}
\usepackage[pdftex]{hyperref} %hyperliens (PDF seulement)
\hypersetup{
pdfauthor={Olivier Rodriguez}, %auteur
pdftitle={
On the Teichm\"uller geodesic generated by the L-shaped translation surface tiled by three squares
}, %titre
pdfsubject={}, %sujet
pdfkeywords={
Riemann surfaces, translation surfaces, algebraic curves, period matrices, Teichm\"uller geodesics
}, %mots-clés
hyperindex=true,
}
\newcommand{\C}{\mathbb{C}}
\newcommand{\I}{\mathrm{I}}
\newcommand{\IM}{\Im\mathrm{m}}
\newcommand{\M}{\mathrm{M}}

\newcommand{\Q}{\mathbb{Q}}
\newcommand{\R}{\mathbb{R}}
\newcommand{\RE}{\Re\mathrm{e}}
\newcommand{\Z}{\mathbb{Z}}
\newcommand*{\T}[1]{{\,\vphantom{#1}}^{\mathit t}{\!#1}}
\newcommand*{\lquotient}[2]{\left.\raisebox{-0.3ex}{$#1$}%
	\backslash\raisebox{0.3ex}{$#2$}\right.}
\newcommand*{\rquotient}[2]{\left.\raisebox{0.3ex}{$#1$}%
	/\raisebox{-0.3ex}{$#2$}\right.}
\makeatletter
\newcommand*\MOD[1]{%
	\allowbreak
	\mkern 6mu
	({\operator@font mod}\,\,#1)%
}\makeatother
\newcommand*{\thetasc}[1]{\vartheta\big[\!\begin{smallmatrix}#1\end{smallmatrix}\!\big]}
\newcommand*{\thetabc}[1]{\vartheta\!\begin{bmatrix}#1\end{bmatrix}\!}
\newcommand*{\scar}[1]{\big[\!\begin{smallmatrix}#1\end{smallmatrix}\!\big]}
\DeclareMathOperator{\AUT}{Aut}
\DeclareMathOperator{\DIAG}{diag}
\DeclareMathOperator{\GL}{GL}
\DeclareMathOperator{\HH}{H}

\DeclareMathOperator{\SL}{SL}
\DeclareMathOperator{\SO}{SO}
\DeclareMathOperator{\SP}{Sp}
\theoremstyle{plain}
\newtheorem{theorem}{Theorem}
\newtheorem{mytheorem}{Theorem}

\newtheorem{proposition}[theorem]{Proposition}

\newtheorem{lemma}[theorem]{Lemma}

\theoremstyle{definition}

\theoremstyle{remark}
\newtheorem{remark}{Remark}
\newtheorem{example}{Example}
\newtheorem*{acknowledgements}{Acknowledgements}
\begin{document}
\renewcommand{\bibname}{References}
\title[L-shaped translation surfaces tiled by three rectangles]{
On the Teichm\"uller geodesic generated by the L-shaped translation surface tiled by three squares
}
\author[O.~Rodriguez]{
Olivier Rodriguez
}
\address{
Institut de Math\'ematiques et de Mod\'elisation de Montpellier\\
UMR~CNRS~5149\\
\hbox{Universit\'e} Montpellier 2 CC~051\\
Place~Eug\`ene Bataillon\\
34095 Montpellier cedex 5\\
France
}
\curraddr{
G\'eosciences Montpellier\\
UMR~CNRS~5243\\
Universit\'e Montpellier 2 CC~MSE\\
300 avenue du Pr.~E.~Jeanbrau\\
34095 Montpellier cedex 5\\
France
}
\email{olivier.rodriguez@univ-montp2.fr}
\date{October 2011}
\begin{abstract}
We study the one parameter family of genus $2$ Riemann surfaces defined by the orbit of the L-shaped translation surface tiled by three squares under the Teichm\"uller geodesic flow.
These surfaces are real algebraic curves with three real components.
We are interested in describing these surfaces by their period matrices.
We show that the only Riemann surface in that family admitting a non-hyperelliptic automorphism comes from the $3$-square-tiled translation surface itself.
This makes the computation of an exact expression for period matrices of other Riemann surfaces in that family by the classical method impossible.
We nevertheless give the solution to the Schottky problem for that family: we exhibit explicit necessary and sufficient conditions for a Riemann matrix to be a period matrix of a Riemann surface in the family, involving the vanishing of a genus $3$ theta characteristic on a family of double covers.
\end{abstract}
\maketitle
\tableofcontents
\section{Introduction}\label{sec:intro}
A translation surface can be defined as an assembling of Euclidean polygons with appropriate identifications of sides or, in an equivalent manner, as a pair $(X,\omega)$ where $X$ is a compact Riemann surface and $\omega$ a holomorphic $1$-form on $X$.
Such a pair can be considered as an element of a rank $g$ vector bundle $\Omega\mathcal{T}_g\to\mathcal{T}_g$ over the Teichm\"uller space $\mathcal{T}_g$ of genus $g$ Riemann surfaces.
The moduli space of holomorphic $1$-forms with a unique (double) zero on a genus $2$ Riemann surface is denoted by $\mathcal{H}(2)$.

There exists a natural action of the group $\GL_2^+(\R)$ on translation surfaces.
The projections of the $\SL_2(\R)$-orbits into the Riemann moduli space $\mathcal{M}_g$ are called Teichm\"uller disks.
It may happen that the stabilizer of a translation surface under the $\SL_2(\R)$-action is a lattice.
Passing to the quotient, it gives rise to a Teichm\"uller curve, that is, an algebraic curve in the Riemann moduli space, isometrically immersed for the Teichm\"uller metric.

To date very few is known about how one passes explicitly from one description of a complex structure to another under the $\GL_2^+(\R)$ action.
For example, how does the period matrix of a Riemann surface vary under this action? What about the equations defining the corresponding algebraic curve?

Note that after this paper was written, M.~M\"oller wrote an important result on a closely related family of examples, see \cite[Theorem~0.1]{moeller11}.\\

In this paper we study the family of Riemann surfaces defined by the $\SL_2(\R)$-orbit of the L-shaped translation surface tiled by three squares (see Figure \ref{fig:3-square-tiled_surface}).

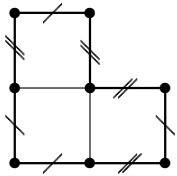
\begin{figure}[!h]
\centering \ifx\JPicScale\undefined\def\JPicScale{1}\fi
\unitlength \JPicScale mm
\begin{picture}(21.25,21.25)(0,0)
\linethickness{0.15mm}
\put(0,10){\line(1,0){10}}
\linethickness{0.15mm}
\put(10,0){\line(0,1){10}}
\linethickness{0.3mm}
\put(10,20){\circle*{1.24}}

\linethickness{0.3mm}
\multiput(-1.25,16.25)(0.12,-0.12){21}{\line(1,0){0.12}}
\linethickness{0.3mm}
\multiput(3.75,-1.25)(0.12,0.12){21}{\line(1,0){0.12}}
\linethickness{0.3mm}
\multiput(8.75,16.25)(0.12,-0.12){21}{\line(1,0){0.12}}
\linethickness{0.3mm}
\multiput(3.75,18.75)(0.12,0.12){21}{\line(1,0){0.12}}
\linethickness{0.3mm}
\multiput(13.75,-1.25)(0.12,0.12){21}{\line(1,0){0.12}}
\linethickness{0.3mm}
\multiput(14.38,-1.25)(0.12,0.12){21}{\line(0,1){0.12}}
\linethickness{0.3mm}
\multiput(13.75,8.75)(0.12,0.12){21}{\line(1,0){0.12}}
\linethickness{0.3mm}
\multiput(13.12,8.75)(0.12,0.12){21}{\line(1,0){0.12}}
\linethickness{0.3mm}
\multiput(8.75,15.62)(0.12,-0.12){21}{\line(1,0){0.12}}
\linethickness{0.3mm}
\multiput(-1.25,16.88)(0.12,-0.12){21}{\line(1,0){0.12}}
\linethickness{0.3mm}
\put(0,0){\circle*{1.24}}

\linethickness{0.3mm}
\put(10,0){\circle*{1.24}}

\linethickness{0.3mm}
\put(20,0){\circle*{1.24}}

\linethickness{0.3mm}
\put(20,10){\circle*{1.24}}

\linethickness{0.3mm}
\put(0,10){\circle*{1.24}}

\linethickness{0.3mm}
\put(10,10){\circle*{1.24}}

\linethickness{0.3mm}
\multiput(-1.25,6.25)(0.12,-0.12){21}{\line(1,0){0.12}}
\linethickness{0.3mm}
\multiput(18.75,6.25)(0.12,-0.12){21}{\line(1,0){0.12}}
\linethickness{0.3mm}
\put(0,20){\circle*{1.24}}

\linethickness{0.3mm}
\put(0,0){\line(0,1){20}}
\linethickness{0.3mm}
\put(0,0){\line(1,0){20}}
\linethickness{0.3mm}
\put(20,0){\line(0,1){10}}
\linethickness{0.3mm}
\put(10,10){\line(1,0){10}}
\linethickness{0.3mm}
\put(10,10){\line(0,1){10}}
\linethickness{0.3mm}
\put(0,20){\line(1,0){10}}
\end{picture}
\caption{The L-shaped translation surface tiled by three squares}\label{fig:3-square-tiled_surface}
\end{figure}

According to \cite[Example in \S6]{mcmullen05a}, this family is the Teichm\"uller curve of discriminant $9$ and, following McMullen's notation, will be denoted by $W_9$ (comprehensive overviews on Teichm\"uller curves can be found {\it e.g.} in \cite{mcmullen03a} and \cite{lochak05}).
We will sometimes denote by $W_9^{\mathrm{M}}$ the Teichm\"uller geodesic generated by the $3$-square-tiled surface, that is, its orbit under the diagonal subgroup \footnote{Note that the usual parametrization of a Teichm\"uller geodesic is $(\begin{smallmatrix}e^t&0\\0&e^{-t}\end{smallmatrix})$ for $t\in\R$.} $\{(\begin{smallmatrix}t&0\\0&t^{-1}\end{smallmatrix})\}_{t>0}$.
The  $3$-square-tiled surface admits an order $4$ automorphism; period matrices of Riemann surfaces defined by such translation surfaces in $\mathcal{H}(2)$ were computed by R.~Silhol in \cite[\S3]{silhol06}.
We will first show the following.

\begin{proposition}\label{prop:autos_w9m}
The only Riemann surface in the family $W_9^{\mathrm{M}}$ admitting a non-hyperelliptic automorphism is the one defined by the L-shaped translation surface tiled by three squares.
\end{proposition}

As automorphisms provide precious informations in order to compute the period matrix of a Riemann surface, that prevents to compute an exact expression for period matrices of these surfaces by the classical method described, for instance, in \cite[\S11.7]{birkenhake&lange04} (see Remark \ref{rema:periods_usual_computation} below).
Nevertheless, we consider a certain family of double covers of those surfaces (whose construction will be described in detail) admitting many automorphisms, for which we obtain the following characterization of their period matrices.

\begin{mytheorem}\label{theo:double_cover_thetanull}
Let $(X,\omega)$ be a translation surface in $\mathcal{H}(2)$, then $X$ is in the family $W_9$ if and only if a certain explicit double cover $\hat{X}$ admits a period matrix $\hat{Z}$ of the form
\[
\hat{Z}=
\begin{pmatrix}
z_1 & \frac12z_1 & z_{13}\\
\frac12z_1 & \frac12+\frac34z_1-\frac12z_{13} & \frac12z_1\\
z_{13} & \frac12z_1 & z_1
\end{pmatrix}
\]
for which $\thetasc{1&1&1\\1&0&1}(\hat{Z})=0$.
\end{mytheorem}

\noindent Recall that for all $m,n\in(\rquotient{\Z}{2\Z})^g$, the theta characteristic $\thetasc{m\\n}$ is defined by
\[
\thetabc{m\\n}(Z)=\sum_{k\in\Z^g} \exp\pi i\left[\T{\left(k+\frac12m\right)}Z\left(k+\frac12m\right)+\T{\left(k+\frac12m\right)}n\right]
\]
for all $Z\in\M_g(\C)$ such that $\T{Z}=Z$ and $\IM(Z)>0$.

Riemann surfaces in the Teichm\"uller geodesic generated by the 3-square-tiled surface correspond to real algebraic curves defined by polynomials with real roots.
The situation being more rigid in this case, this allows us to establish a form for their period matrices and deduce necessary and sufficient conditions for a Riemann surface to be in the family $W_9^{\mathrm{M}}$.

\begin{mytheorem}\label{theo:w9m_periods}	
Let $(X_1,\omega_1)$ be the L-shaped translation surface tiled by three squares.
For any real number $t\geq1$, let $(X_t,\omega_t)=\left(\begin{smallmatrix}1&0\\ 0&t\end{smallmatrix}\right)\cdot(X_1,\omega_1)$.
Then $Z_t$ is a period matrix of $X_t$ associated to a certain explicit homology basis if, and only if there exists a unique real number $y_t>2t/3$ such that
\[
Z_t=
\begin{pmatrix}
1+i(2y_t-t) & iy_t\\
iy_t & i(y_t/2+t)
\end{pmatrix}
\]
and satisfying
\begin{multline}\label{eq:main_equation}
\sum_{(k_1,k_2,k_3)\in\Z^3}\exp\pi\Biggl[\\
\left(\frac{t}{2}-y_t+\frac{1}{2}i\right)\sum_{\l=1}^3 k_\ell^2+\left(y_t-t+i\right)\sum_{1\leq\ell<m\leq3}k_\ell k_m+\left(\frac{3}{2}i-\frac{t}{2}\right)\sum_{\ell=1}^3 k_\ell\Biggr]=0.
\end{multline}
\end{mytheorem}

\noindent The construction of the homology basis will be described in detail.
\begin{acknowledgements}
I would like to express my gratitude toward my Ph.D. advisor Robert Silhol for helpful conjectures as well as fulfilling discussions.
I also wish to thank Emmanuel Royer, Pascal Hubert and Guillaume Bulteau for various advices and useful discussions and comments.
\end{acknowledgements}
%%%%%%%%%%%%%%%%%%%%%%%%%%%%%%
%
\section{Preliminaries}\label{sec:preliminaries}
\subsection{Translation surfaces}\label{subsec:translation_surfaces}
A \emph{translation surface} is a finite collection of Euclidean polygons in the complex plane such that\begin{itemize}
\item  the boundary of each polygon is oriented counterclockwise;
\item for every side of a polygon, there exists another side (possibly of the same polygon) parallel and of the same length: both sides are then identified by translation.
\end{itemize}
Such a collection of polygons defines a topological surface admitting a \emph{translation structure}, that is, away from a finite set of points, a maximal atlas whose transition functions are translations.

As a non-trivial example consider a compact Riemann surface equipped with a holomorphic $1$-form.
Integrating the form yields, away from its zeros, an atlas with polygonal charts and transition functions that are translations.
Conversely, a translation structure defines a complex structure since translations are biholomorphic.
Pulling back the $1$-form $dz$ on the complex plane by the charts gives a holomorphic $1$-form on the surface.

As a consequence, we can define in an equivalent manner a translation surface as a pair $(X,\omega)$ where $X$ is a compact Riemann surface and $\omega$ a holomorphic $1$-form on $X$.
At a regular point, in the local coordinate defined by integrating the form, we have $\omega=dz$.
At an order $k$ zero of $\omega$, we have
\[
\omega=z^kdz=d\left(\frac{z^{k+1}}{k+1}\right)
\]
so that the Riemann surface $X$ is locally a ($k+1$)-fold cover over the complex plane.
This means that an order $k$ zero corresponds to a cone-type singularity of angle $2\pi(k+1)$ for the locally Euclidean metric $|\omega|$.
More details can be found in \cite[\S1]{masur06} concerning the equivalence of these definitions.
See also \cite{zorich06} for a general survey on translation surfaces.

We will use the notation $(X,\omega)=(\rquotient{\mathcal{P}}{\sim},dz)$ where the quotient $\rquotient{\mathcal{P}}{\sim}$ designates an assembling of Euclidean polygons with appropriate identifications of sides by translation and $dz$ is the holomorphic $1$-form on $\C$.

\begin{example}[Translation surface tiled by three squares]\label{exam:3-square-tiled_surface}
Consider the L-shaped polygon obtained by assembling three copies of the Euclidean unit square equipped with the glueings specified in Figure \ref{fig:3-square-tiled_surface}.
The black dots are identified to the cone-type singularity of angle $6\pi$, hence defining a translation surface $(X,\omega)$ in $\mathcal{H}(2)$.
Rotation by angle $\pi/2$ around the center of the bottom left square induces an order $4$ automorphism on the Riemann surface $X$: the corresponding algebraic curve then admits an equation of the form
\[
y^2=x(x^2-1)(x-a)(x-1/a)
\]
and the order $4$ automorphism is
\[
(x,y)\mapsto\left(\frac1x, \frac{iy}{x^3}\right).
\]
Following \cite[\S3\&4]{silhol06}, in this example we have $a=7+4\sqrt3$ and the holomorphic $1$-form $\omega$ is
\[
\omega=\mu\left(\frac{dx}{y}-\frac{xdx}{y}\right)
\]
with $\mu\in\C^*$.
\end{example}

\begin{remark}\label{rema:half-translation_surface}
From a more general point of view, translation surfaces are a specific case of half-translation surfaces, for which the transition functions are of the form $z\mapsto\pm z+c$ for $c\in\C$.
Such a surface can be defined as a pair $(X,q)$ where $q$ is a holomorphic quadratic differential on $X$.
If $(X,\omega)$ is a translation surface, then the quadratic differential defining the complex structure is $\omega^2$.
\end{remark}
%%%%%%%%%%
%
\subsection{$\GL_2^+(\R)$-action}\label{subsec:GL(2,R)-action}
There exists a natural action of the linear group on translation structures.
An element $M=(\begin{smallmatrix}a&b\\c&d\end{smallmatrix})\in\GL_2^+(\R)$ acts on $z\in\C$ by
\[
M\cdot z=ax+by+i(cx+dy).
\]
 This is just the affine action on the complex plane identified to $\R^2$.
When a translation surface is defined by a collection of polygons $(\rquotient{\mathcal{P}}{\sim},dz)$, then the group $\GL_2^+(\R)$ operates naturally on the polygons, giving a new translation surface.
The action is well defined since linear applications transform parallelograms into parallelograms.

On a form $(X,\omega)$, the action of $M$ is defined as follows: let
\[
\eta:=a\RE(\omega)+b\IM(\omega)+i\big(c\RE(\omega)+d\IM(\omega)\big),
\]
then $\eta$ is a harmonic form on $X$.
There exists a unique complex structure on the underlying topological surface for which $\omega$ is holomorphic, so that we obtain a new Riemann surface $Y$ and we set
\[
M\cdot(X,\omega)=(Y,\eta).
\]

\begin{remark}\label{rema:stable_under_SO(2,R)}
For every translation surface $(X,\omega)$, the complex structure defined by $(X, \omega)$ is stable under the action of the subgroup $\R_+^*\cdot\SO_2(\R$), since this action corresponds to multiplying $\omega$ by a non zero scalar complex.
\end{remark}
%%%%%%%%%%
\subsubsection*{Teichm\"uller geodesics}\label{paragraph:teichmueller_geodesics}
Let $t$ be a real number such that $t>1$ and consider the $\R$-linear application $M_t:\C\to\C$ defined by
\[
z\mapsto\frac12(1+t)z+\frac12(1-t)\bar{z}.
\]
Let $(X,\omega)$ be a translation surface and $(X_t,\omega_t)=M_t\cdot(X,\omega)$.
The change of complex structure yields a natural application $f_t:X\to X_t$ verifying $df_t=M_t$ and having constant complex dilatation:
\[
\forall P\in X,\quad\mu_{f_t}(P):=\frac{\bar{\partial}f_t(P)}{\partial f_t(P)}=\frac{1-t}{1+t}.
\]
The maximal dilatation is then
\[
K(f_t):=\sup_{P\in X}\frac{1+|\mu_{f_t}(P)|}{1-|\mu_{f_t}(P)|}=t.
\]
The homeomorphism $f_t:X\to X_t$ is then a Teichm\"uller extremal map.
We call the Riemann surface $X_t$ the \emph{Teichm\"uller deformation of $X$ of dilatation $t$ with respect to $\omega$} and the family $\{X_t\}_{t\geq1}$ a \emph{Teichm\"uller geodesic} (see {\it e.g.} \cite[chap.~1]{abikoff80} for more details).

\begin{remark}\label{rema:isomorphic_surfaces}
Let $(X,\omega)$ be a L-shaped translation surface of the form indicated in Figure \ref{fig:L-shaped_surface}.
Then $X$ admits an order $4$ automorphism induced by rotation of angle $\pi/2$ around $0$.
\begin{figure}[!h]
\centering \ifx\JPicScale\undefined\def\JPicScale{1}\fi
\unitlength \JPicScale mm
\begin{picture}(40,40)(0,0)
\linethickness{0.15mm}
\put(2.5,5){\line(0,1){30}}
\put(2.5,35){\vector(0,1){0.12}}
\put(2.5,5){\vector(0,-1){0.12}}
\linethickness{0.15mm}
\put(5,2.5){\line(1,0){30}}
\put(35,2.5){\vector(1,0){0.12}}
\put(5,2.5){\vector(-1,0){0.12}}
\put(20,0){\makebox(0,0)[cc]{$\lambda$}}

\put(0,20){\makebox(0,0)[cc]{$\lambda$}}

\linethickness{0.3mm}
\put(5,5){\line(1,0){30}}
\linethickness{0.3mm}
\put(15,15){\line(1,0){15}}
\linethickness{0.3mm}
\put(15,30){\line(0,1){5}}
\put(15,30){\vector(0,-1){0.12}}
\linethickness{0.3mm}
\put(5,5){\line(0,1){30}}
\linethickness{0.3mm}
\put(5,35){\line(1,0){10}}
\linethickness{0.3mm}
\put(35,5){\line(0,1){10}}
\linethickness{0.3mm}
\put(5,35){\circle*{1.24}}

\linethickness{0.3mm}
\put(5,15){\circle*{1.24}}

\linethickness{0.3mm}
\put(15,35){\circle*{1.24}}

\linethickness{0.3mm}
\put(15,15){\circle*{1.24}}

\linethickness{0.3mm}
\put(35,15){\circle*{1.24}}

\linethickness{0.3mm}
\put(5,5){\circle*{1.24}}

\linethickness{0.3mm}
\put(15,5){\circle*{1.24}}

\linethickness{0.3mm}
\put(35,5){\circle*{1.24}}

\linethickness{0.3mm}
\multiput(3.75,11.25)(0.12,-0.12){21}{\line(1,0){0.12}}
\linethickness{0.3mm}
\multiput(33.75,11.25)(0.12,-0.12){21}{\line(1,0){0.12}}
\linethickness{0.3mm}
\multiput(23.75,13.75)(0.12,0.12){21}{\line(1,0){0.12}}
\linethickness{0.3mm}
\multiput(23.12,13.75)(0.12,0.12){21}{\line(1,0){0.12}}
\linethickness{0.3mm}
\multiput(23.75,3.75)(0.12,0.12){21}{\line(1,0){0.12}}
\linethickness{0.3mm}
\multiput(24.38,3.75)(0.12,0.12){21}{\line(1,0){0.12}}
\linethickness{0.3mm}
\multiput(8.75,3.75)(0.12,0.12){21}{\line(1,0){0.12}}
\linethickness{0.3mm}
\multiput(8.75,33.75)(0.12,0.12){21}{\line(1,0){0.12}}
\linethickness{0.3mm}
\multiput(3.75,26.25)(0.12,-0.12){21}{\line(1,0){0.12}}
\linethickness{0.3mm}
\multiput(13.75,26.25)(0.12,-0.12){21}{\line(1,0){0.12}}
\linethickness{0.3mm}
\multiput(13.75,25.62)(0.12,-0.12){21}{\line(1,0){0.12}}
\linethickness{0.3mm}
\multiput(3.75,26.88)(0.12,-0.12){21}{\line(1,0){0.12}}
\linethickness{0.15mm}
\put(5,10){\line(1,0){17.5}}
\put(22.5,10){\vector(1,0){0.12}}
\linethickness{0.15mm}
\put(22.5,10){\line(1,0){12.5}}
\linethickness{0.3mm}
\put(30,15){\line(1,0){5}}
\put(30,15){\vector(-1,0){0.12}}
\linethickness{0.3mm}
\put(15,15){\line(0,1){15}}
\linethickness{0.15mm}
\put(10,20){\line(0,1){15}}
\linethickness{0.15mm}
\put(10,5){\line(0,1){15}}
\put(10,20){\vector(0,1){0.12}}
\put(17.5,30){\makebox(0,0)[cc]{$\beta_2$}}

\put(7.5,20){\makebox(0,0)[cc]{$\beta_1$}}

\put(22.5,7.5){\makebox(0,0)[cc]{$\alpha_1$}}

\put(30,17.5){\makebox(0,0)[cc]{$\alpha_2$}}

\put(11.25,12.5){\makebox(0,0)[cc]{$0$}}

\linethickness{0.3mm}
\put(9.38,10){\line(1,0){1.24}}
\linethickness{0.3mm}
\put(10,9.38){\line(0,1){1.24}}
\put(30,30){\makebox(0,0)[cc]{$(X,\omega)$}}

\linethickness{0.15mm}
\put(37.5,5){\line(0,1){10}}
\put(37.5,15){\vector(0,1){0.12}}
\put(37.5,5){\vector(0,-1){0.12}}
\linethickness{0.15mm}
\put(5,37.5){\line(1,0){10}}
\put(15,37.5){\vector(1,0){0.12}}
\put(5,37.5){\vector(-1,0){0.12}}
\put(40,10){\makebox(0,0)[cc]{$1$}}

\put(10,40){\makebox(0,0)[cc]{$1$}}

\end{picture}
\caption{An L-shaped translation surface with an order $4$ automorphism}\label{fig:L-shaped_surface}
\end{figure}
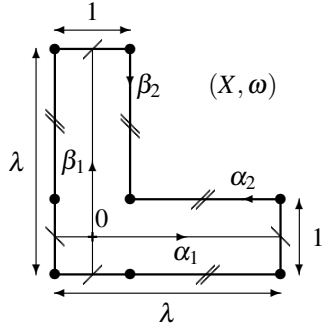
Let $t$ be a real number such that $t>0$, we define the translation surfaces
\begin{align*}
(X_{g_t},\omega_{g_t})&=(\begin{smallmatrix}t&0\\0&t^{-1}\end{smallmatrix})\cdot(X,\omega),\\
(X_{h_t},\omega_{h_t})&=(\begin{smallmatrix}t&0\\0&1\end{smallmatrix})\cdot(X,\omega),\\
\text{and }(X_{v_t},\omega_{v_t})&=(\begin{smallmatrix}1&0\\0&t\end{smallmatrix})\cdot(X,\omega).
\end{align*}
Then by Remark \ref{rema:stable_under_SO(2,R)} it is readily checked that
\begin{enumerate}[(1)]
\item the Riemann surfaces $X_{h_t}$ and $X_{v_t}$ are isomorphic;
\item $X_{g_t}$ and $X_{h_{t^2}}$ (resp. $X_{g_{t^{-1}}}$ and $X_{v_{t^2}}$) are isomorphic;
\item $X_{h_t}$ and $X_{h_{t^{-1}}}$ (resp. $X_{v_t}$ and $X_{v_{t^{-1}}}$) are isomorphic.
\end{enumerate}
\end{remark}
%%%%%%%%%%
%
\subsection{Period matrix}\label{subsec:period_matrix}
Let $X$ be a compact Riemann surface of genus $g$ and let $\mathcal{B}=(\alpha_1,\ldots,\alpha_g,\beta_1,\ldots,\beta_g)$ be a symplectic basis for $\HH_1(X,\Z)$, that is, such that the matrix of the intersection product is
\[
\begin{pmatrix}0&-\I_g\\ \I_g&0\end{pmatrix}
\]
Let $(\omega_1,\ldots,\omega_g)$ be a basis of holomorphic $1$-forms on $X$, we define
\[
A=\left(\int_{\alpha_j}\omega_k\right)_{j,k}\text{ and }B=\left(\int_{\beta_j}\omega_k\right)_{j,k}.
\]
Then $Z=AB^{-1}$ is \emph{the period matrix associated to $\mathcal{B}$}.
Note that $Z$ does not depend on the choice of the basis $(\omega_1,\ldots,\omega_g$).

It is well known that $Z$ verifies the Riemann bilinear relations, that is, $Z$ is symmetric and $\IM(Z)$ is positive definite.
The Siegel upper half-space is defined by
\[
\mathfrak{S}_g=\{Z\in\M_g:\T{Z}=Z\text{ and }\IM(Z)>0\}.
\]

Let $\alpha,\beta,\gamma,\delta\in\M_g(\R)$, an element $M=(\begin{smallmatrix}\alpha&\beta\\ \gamma&\delta\end{smallmatrix})$  of the symplectic group
\[
\SP_{2g}(\R)=\left\{M\in\M_{2g}(\R):\T{M}
\begin{pmatrix}0&-\I_g\\ \I_g&0\end{pmatrix}
M=
\begin{pmatrix}0&-\I_g\\ \I_g&0\end{pmatrix}\right\}
\]
acts on an element $Z\in\frak{S}_g$ as follows:
\[
M(Z)=(\alpha Z+\beta)(\gamma Z+\delta)^{-1}.
\]

If $\mathcal{B}'$ is another symplectic basis for $\HH_1(X,\Z)$ and $Z'$ the period matrix associated to $\mathcal{B}'$, let $M$ be the base change matrix from $\mathcal{B}'$ to $\mathcal{B}$; then $M$ is a symplectic matrix with integer entries and $Z$ and $Z'$ are related by
\[
Z'=\T{M}(Z).
\]
Recall that if $M$ is symplectic, then $\T{M}$ is also symplectic.

Moreover, by the Torelli theorem, the period matrix characterizes the complex structure: if $X$ and $X'$ are compact Riemann surfaces with period matrices $Z$ and $Z'$ respectively, then $X$ and $X'$ are isomorphic if, and only if there exists $M\in\SP_{2g}(\Z)$ such that $M(Z)=Z'$ (see \cite[Theorem~11.1.7]{birkenhake&lange04}).

\begin{example}[Translation surface with an order $4$ automorphism]\label{exam:order4_automorphism}
Let $(X,\omega)$ be the L-shaped translation surface defined in Figure \ref{fig:L-shaped_surface}.
The rotation by angle $\pi/2$ around $0$ induces an order $4$ automorphism on the Riemann surface $X$.
Then following \cite[\S3]{silhol06},
\[
Z=i\begin{pmatrix}\frac{2\lambda^2-2\lambda+1}{2\lambda-1}&\frac{-2\lambda(\lambda-1)}{2\lambda-1}\\ \frac{-2\lambda(\lambda-1)}{2\lambda-1}&\frac{2\lambda^2-2\lambda+1}{2\lambda-1}\end{pmatrix}
\]
is the period matrix of $X$ associated to $\{\alpha_1,\alpha_2,\beta_1,\beta_2\}$.
\end{example}

\begin{remark}\label{rema:periods_usual_computation}
The main tool in Silhol's calculation is the existence of a non-hyperelliptic involution whose action permutes $\omega$ with another linearly independent holomorphic $1$-form that can be explicitly described (enabling to compute the whole period matrix, following \cite[\S11.7]{birkenhake&lange04} for example).
However, on the one hand, generic algebraic curves have no non-trivial automorphisms, and on the other hand it is not known in general how to provide a description of another holomorphic $1$-form.
\end{remark}
%%%%%%%%%%
\subsubsection*{Hyperelliptic Riemann surfaces}\label{paragraph:hyperelliptic_homology_basis}
Here we describe a construction of a period matrix of a hyperelliptic Riemann surface.
If $X$ is hyperelliptic, then the associated algebraic curve is defined by an equation of the form
\[
y^2=\prod_{j=1}^m(x-x_j)
\]
where the $x_j$ are distinct complex numbers and $m=2g+1$ or $2g+2$.
If $m=2g+1$, then let $x_{2g+2}=\infty$ 

Let $\varepsilon_1$ be a simple arc in $\mathbb{P}^1(\C)$ joining $x_1$ to $x_2$ and not passing through any of the other $x_j$ for $j\geq3$.
Let $\varepsilon_2$ a second simple arc joining $x_2$ to $x_3$ and not passing through  $x_1$ nor any of the other $x_j$'s for $j\geq4$, and such that  $\varepsilon_2$ only intersects $\varepsilon_1$ in $x_2$.
In the same way, we construct simple arcs $\varepsilon_3,\ldots,\varepsilon_{2g+1},\varepsilon_{2g+2}$ joining respectively $x_3$ to $x_4$, \ldots, $x_{2g+1}$ to $x_{2g+2}$ and $x_{2g+2}$ to $x_1$, so that $\varepsilon_j$ only intersects $\varepsilon_{j+1}$ in one point ($j\mod{2g+2}$).
Let $\pi:X\to\mathbb{P}^1(\C)$ denote the projection $(x,y)\mapsto x$ and, for $j=1,\ldots,2g+2$, let $\delta_j:=\pi^{-1}(\varepsilon_j)$: this is a simple closed curve in $X$.

We choose on each $\varepsilon_j$ a holomorphic determination of $\sqrt{P}$ so that the induced orientation on $\delta_j$ is such that $(\delta_j\cdot\delta_{j+1})=1$ ($j\mod2g+2)$, all other intersection numbers being zero.
With this convention, up to homology, we thus have $\sum_{j=1}^g\delta_{2j}=-\delta_{2g+2}$ and $\sum_{j=1}^g\delta_{2j-1}=-\delta_{2g+1}$ (this is a readily checked by a topological sketch of the situation, the reader may also refer to \cite[\S VII.1.1]{farkas&kra92} where this construction is described).

For $j=1,\ldots,g$, we define
\begin{align*}
\alpha_j &=-\sum_{\ell=1}^{j-1}\delta_{2\ell}-\delta_{2g+2},\\
\beta_j &=\delta_{2j-1}.
\end{align*}
It is readily checked that $\mathcal{B}:=(\alpha_1,\ldots,\alpha_g,\beta_1,\ldots,\beta_g)$ is a symplectic basis for $\HH_1(X,\Z)$.
Moreover, as $X$ is a hyperelliptic Riemann surface, 
\[
\frac{dx}{y},\frac{xdx}{y},\ldots,\frac{x^{g-1}dx}{y}
\]
is a basis of holomorphic $1$-form on $X$ (see \cite[\S III.7.5, Corollary~1]{farkas&kra92}).
Let
\begin{align*}
A&=\left(-\sum_{\ell=1}^{j-1}\int_{\varepsilon_{2\ell}}\frac{x^{k-1}dx}{\sqrt{P(x)}}-\int_{\varepsilon_{2g+2}}\frac{x^{k-1}dx}{\sqrt{P(x)}}\right)_{j,k}\\
B&=\left(\int_{\varepsilon_{2j-1}}\frac{x^{k-1}dx}{\sqrt{P(x)}}\right)_{j,k}
\end{align*}
then $Z=AB^{-1}$ is the period matrix associated to $\mathcal{B}$.
%%%%%%%%%%
%
\subsection{Real algebraic curves}\label{subsec:real_alg_curves}
\subsubsection*{Definitions}\label{paragraph:real_alg_curves_defs}
F.~Klein observed that a complex algebraic curve $X$ is defined by real polynomial equations if, and only if $X$ admits an anti-holomorphic involution $\sigma$ (see \cite[\S21]{klein63}), called a \emph{real structure}.
Moreover, these polynomials can always be chosen so that $\sigma$ is induced by complex conjugation.
A \emph{real algebraic curve} is a couple $(X,\sigma)$.
When it is clear from the context, we omit $\sigma$ and simply say that $X$ is a real curve, in particular when $X$ is defined by a real polynomial equation and $\sigma$ is the complex conjugation.

If $X$ is of genus $g$, then the connected components of the fixed point set of $\sigma$ are said to be real.
We say that $X$ is an M-curve if it admits the maximum number of real components, which is $g+1$ (see \cite[Proposition~3.1]{gross&harris81}).
%%%%%%%%%
\subsubsection*{Real hyperelliptic M-curves}\label{paragraph:real_hyperelliptic_mcurves}
Suppose that $X$ is a hyperelliptic curve defined by a polynomial equation $y^2=P(x)$ with $\deg P=2g+1$ or $2g+2$, and such that $\sigma$ is induced by complex conjugation.
Then $(X,\sigma)$ is a real M-curve if, and only if all roots of $P$ are real.
Composing $\sigma$ with the hyperelliptic involution $h_X$ yields a second real structure denoted by $-\sigma$.
The connected components of the fixed point set of $-\sigma$ are said to be pure imaginary and the Weierstrass points of $X$ are exactly the intersection points of the real and pure imaginary components.

\begin{example}[M-curve defined by a translation surface]\label{exam:mcurve_from_translation_surface}
Let $(X,\omega)=(\rquotient{\mathcal{P}}{\sim},dz)$ be a translation surface in $\mathcal{H}(2)$ obtained from four mirror images of a L-shaped polygon (see Figure \ref{fig:mcurve}).
We can always assume that the polygon $\mathcal{P}$ admits $0$ as a center of symmetry and is stable by complex conjugation.
The latter defines an anti-holomorphic involution $\sigma$ on $X$ that fixes pointwise the three simple closed curves coming from the horizontal axis of symmetry of $\mathcal{P}$ and its horizontal sides: we thus obtain a real M-curve $(X,\sigma)$.
The Weierstrass points are represented by black dots (identified to the double zero of $\omega$) and small circles in Figure \ref{fig:mcurve}.
Cutting and reassembling the cross-shaped polygon in the left part of Figure \ref{fig:mcurve} gives the L-shaped polygon in the right part of Figure \ref{fig:mcurve}, where the segments corresponding to the simple closed curves fixed by $\sigma$ are the horizontal and vertical segments that join the black dots and those passing through the circles.

\begin{figure}[!ht]
\centering \ifx\JPicScale\undefined\def\JPicScale{1}\fi
\unitlength \JPicScale mm
\begin{picture}(111.25,48.75)(0,0)
\linethickness{0.3mm}
\put(2.5,12.5){\line(0,1){15}}
\linethickness{0.3mm}
\put(2.5,27.5){\line(1,0){12.5}}
\linethickness{0.3mm}
\put(2.5,12.5){\line(1,0){12.5}}
\linethickness{0.3mm}
\put(15,27.5){\line(0,1){10}}
\linethickness{0.3mm}
\put(15,2.5){\line(0,1){10}}
\linethickness{0.3mm}
\put(15,37.5){\line(1,0){20}}
\linethickness{0.3mm}
\put(15,2.5){\line(1,0){20}}
\linethickness{0.3mm}
\put(35,27.5){\line(0,1){10}}
\linethickness{0.3mm}
\put(35,2.5){\line(0,1){10}}
\linethickness{0.3mm}
\put(47.5,12.5){\line(0,1){15}}
\linethickness{0.3mm}
\put(35,27.5){\line(1,0){12.5}}
\linethickness{0.3mm}
\put(35,12.5){\line(1,0){12.5}}
\linethickness{0.2mm}
\put(25,2.5){\line(0,1){35}}
\linethickness{0.2mm}
\put(2.5,20){\line(1,0){45}}
\linethickness{0.3mm}
\put(65,47.5){\line(1,0){20}}
\linethickness{0.2mm}
\put(75,12.5){\line(0,1){35}}
\linethickness{0.2mm}
\multiput(65,37.5)(1.9,0){11}{\line(1,0){0.95}}
\linethickness{0.3mm}
\put(85,27.5){\line(0,1){20}}
\linethickness{0.3mm}
\put(85,27.5){\line(1,0){25}}
\linethickness{0.3mm}
\put(110,12.5){\line(0,1){15}}
\linethickness{0.2mm}
\put(65,20){\line(1,0){45}}
\linethickness{0.3mm}
\put(65,12.5){\line(1,0){45}}
\linethickness{0.3mm}
\put(65,12.5){\line(0,1){35}}
\linethickness{0.2mm}
\multiput(97.5,12.5)(0,2){8}{\line(0,1){1}}
\linethickness{0.3mm}
\multiput(73.75,11.25)(0.12,0.12){21}{\line(1,0){0.12}}
\linethickness{0.3mm}
\multiput(95.62,11.25)(0.12,0.12){21}{\line(1,0){0.12}}
\linethickness{0.3mm}
\multiput(96.25,26.25)(0.12,0.12){21}{\line(1,0){0.12}}
\linethickness{0.3mm}
\multiput(96.88,26.25)(0.12,0.12){21}{\line(1,0){0.12}}
\linethickness{0.3mm}
\multiput(73.75,46.25)(0.12,0.12){21}{\line(1,0){0.12}}
\linethickness{0.3mm}
\multiput(63.75,21.25)(0.12,-0.12){21}{\line(1,0){0.12}}
\linethickness{0.3mm}
\multiput(83.75,38.12)(0.12,-0.12){21}{\line(1,0){0.12}}
\linethickness{0.3mm}
\multiput(63.75,39.38)(0.12,-0.12){21}{\line(1,0){0.12}}
\linethickness{0.3mm}
\multiput(63.75,38.75)(0.12,-0.12){21}{\line(1,0){0.12}}
\linethickness{0.3mm}
\multiput(108.75,21.25)(0.12,-0.12){21}{\line(1,0){0.12}}
\linethickness{0.3mm}
\multiput(83.75,38.75)(0.12,-0.12){21}{\line(1,0){0.12}}
\linethickness{0.3mm}
\multiput(96.25,11.25)(0.12,0.12){21}{\line(1,0){0.12}}
\linethickness{0.2mm}
\multiput(15,12.5)(0,2){8}{\line(0,1){1}}
\linethickness{0.2mm}
\multiput(15,12.5)(1.9,0){11}{\line(1,0){0.95}}
\put(26.25,22.5){\makebox(0,0)[cc]{$0$}}

\linethickness{0.3mm}
\multiput(1.25,21.25)(0.12,-0.12){21}{\line(1,0){0.12}}
\linethickness{0.3mm}
\multiput(46.25,21.25)(0.12,-0.12){21}{\line(1,0){0.12}}
\linethickness{0.3mm}
\multiput(23.75,36.25)(0.12,0.12){21}{\line(1,0){0.12}}
\linethickness{0.3mm}
\multiput(23.75,1.25)(0.12,0.12){21}{\line(1,0){0.12}}
\linethickness{0.3mm}
\multiput(33.75,33.75)(0.12,-0.12){21}{\line(1,0){0.12}}
\linethickness{0.3mm}
\multiput(33.75,33.12)(0.12,-0.12){21}{\line(1,0){0.12}}
\linethickness{0.3mm}
\multiput(13.75,34.38)(0.12,-0.12){21}{\line(0,-1){0.12}}
\linethickness{0.3mm}
\multiput(13.75,33.75)(0.12,-0.12){21}{\line(1,0){0.12}}
\linethickness{0.3mm}
\multiput(7.5,26.25)(0.12,0.12){21}{\line(1,0){0.12}}
\linethickness{0.3mm}
\multiput(8.12,26.25)(0.12,0.12){21}{\line(1,0){0.12}}
\linethickness{0.3mm}
\multiput(7.5,11.25)(0.12,0.12){21}{\line(1,0){0.12}}
\linethickness{0.3mm}
\multiput(6.88,11.25)(0.12,0.12){21}{\line(1,0){0.12}}
\linethickness{0.3mm}
\multiput(40,26.25)(0.12,0.12){21}{\line(1,0){0.12}}
\linethickness{0.3mm}
\multiput(39.38,26.25)(0.12,0.12){21}{\line(1,0){0.12}}
\linethickness{0.3mm}
\multiput(40.62,26.25)(0.12,0.12){21}{\line(1,0){0.12}}
\linethickness{0.3mm}
\multiput(39.37,11.25)(0.12,0.12){21}{\line(1,0){0.12}}
\linethickness{0.3mm}
\multiput(40.62,11.25)(0.12,0.12){21}{\line(1,0){0.12}}
\linethickness{0.3mm}
\multiput(40,11.25)(0.12,0.12){21}{\line(1,0){0.12}}
\linethickness{0.3mm}
\multiput(13.75,8.75)(0.12,-0.12){21}{\line(1,0){0.12}}
\linethickness{0.3mm}
\multiput(33.75,8.75)(0.12,-0.12){21}{\line(1,0){0.12}}
\linethickness{0.3mm}
\multiput(33.75,9.38)(0.12,-0.12){21}{\line(0,-1){0.12}}
\linethickness{0.3mm}
\multiput(13.75,9.38)(0.12,-0.12){21}{\line(0,-1){0.12}}
\linethickness{0.3mm}
\multiput(33.75,8.12)(0.12,-0.12){21}{\line(1,0){0.12}}
\linethickness{0.3mm}
\multiput(13.75,8.12)(0.12,-0.12){21}{\line(1,0){0.12}}
\linethickness{0.3mm}
\put(65,37.5){\circle{1.88}}

\linethickness{0.3mm}
\put(75,37.5){\circle{1.88}}

\linethickness{0.3mm}
\put(85,37.5){\circle{1.88}}

\linethickness{0.3mm}
\put(75,20){\circle{1.88}}

\linethickness{0.3mm}
\put(97.5,27.5){\circle{1.88}}

\linethickness{0.3mm}
\put(97.5,20){\circle{1.88}}

\linethickness{0.3mm}
\put(97.5,12.5){\circle{1.88}}

\linethickness{0.3mm}
\put(2.5,27.5){\circle{1.88}}

\linethickness{0.3mm}
\put(2.5,20){\circle{1.88}}

\linethickness{0.3mm}
\put(2.5,12.5){\circle{1.88}}

\linethickness{0.3mm}
\put(25,2.5){\circle{1.88}}

\linethickness{0.3mm}
\put(15,2.5){\circle{1.88}}

\linethickness{0.3mm}
\put(35,2.5){\circle{1.88}}

\linethickness{0.3mm}
\put(47.5,20){\circle{1.88}}

\linethickness{0.3mm}
\put(47.5,12.5){\circle{1.88}}

\linethickness{0.3mm}
\put(47.5,27.5){\circle{1.88}}

\linethickness{0.3mm}
\put(25,37.5){\circle{1.88}}

\linethickness{0.3mm}
\put(35,37.5){\circle{1.88}}

\linethickness{0.3mm}
\put(15,37.5){\circle{1.88}}

\linethickness{0.3mm}
\put(25,20){\circle{1.88}}

\linethickness{0.3mm}
\put(65,47.5){\circle*{1.24}}

\linethickness{0.3mm}
\put(85,47.5){\circle*{1.24}}

\linethickness{0.3mm}
\put(85,27.5){\circle*{1.24}}

\linethickness{0.3mm}
\put(65,27.5){\circle*{1.24}}

\linethickness{0.3mm}
\put(110,12.5){\circle*{1.24}}

\linethickness{0.3mm}
\put(110,27.5){\circle*{1.24}}

\linethickness{0.3mm}
\put(65,12.5){\circle*{1.24}}

\linethickness{0.3mm}
\put(85,12.5){\circle*{1.24}}

\linethickness{0.3mm}
\put(35,27.5){\circle*{1.24}}

\linethickness{0.3mm}
\put(15,27.5){\circle*{1.24}}

\linethickness{0.3mm}
\put(15,12.5){\circle*{1.24}}

\linethickness{0.3mm}
\put(35,12.5){\circle*{1.24}}

\put(76.25,22.5){\makebox(0,0)[cc]{$0$}}

\end{picture}
\caption{A L-shaped translation surface defines a real M-curve}\label{fig:mcurve}
\end{figure}
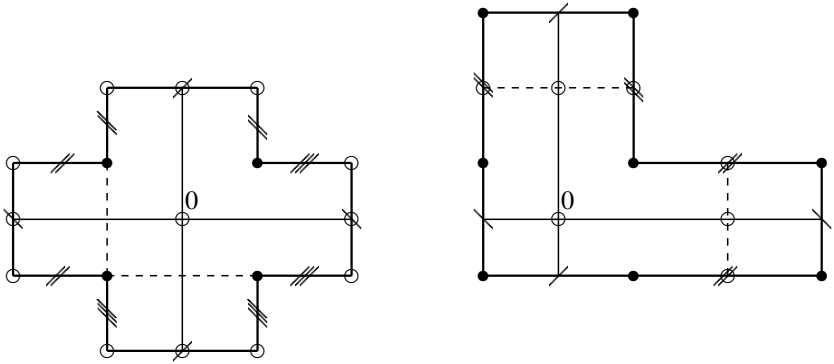
\end{example}

\begin{remark}\label{rema:mcurves_are_teich_def}
In what follows, if $(X,\sigma)$ is a real M-curve defined by a translation surface $(X,\omega)$, then the real structure $\sigma$ will always be the one defined by complex conjugation as in Example \ref{exam:mcurve_from_translation_surface}.
If $(X_1,\omega_1)$ denotes the L-shaped translation surface tiled by three squares, then it corresponds to the unique, so-called \emph{splitting prototype} in \cite[\S3]{mcmullen05a} and by \cite[\S6]{mcmullen05a}, this translation surface generates the whole Teichm\"uller curve $W_9$.
The above convention on the definition of the real structure fixes the orientation of the L-shaped polygon, then by Remark \ref{rema:isomorphic_surfaces} the set of M-curves in the family $W_9$ thus coincides with the Teichm\"uller deformations of $X_1$ with respect to $\omega_1$.
\end{remark}
%%%%%%%%%%
\subsubsection*{Automorphisms of genus $2$ real M-curves}\label{paragraph:genus2_real_mcurves_autos}
If $(X,\sigma)$ is a real algebraic curve, we denote by $\AUT(X,\sigma)$ the group of its \emph{real automorphisms},  that is, biholomorphic applications $\phi:X\to X$ such that $\phi\circ\sigma=\sigma\circ\phi$.
The group $\AUT X$ is sometimes called the group of \emph{complex automorphisms} of $(X,\sigma)$.

F.-J.~Cirre computed the groups of real and complex automorphisms of a real M-curve as stated below, where $D_n$ denote the dihedral group of order $2n$.

\begin{theorem}[Cirre]\label{theo:cirre}
Let $a$, $b$, $c$ be three real numbers such that $0<a<b<c<1$ and let $(X,\sigma)$ be the real M-curve defined by the equation
\[
y^2=P(x)=x(x-a)(x-b)(x-c)(x-1)
\]
Then $\AUT(X,\sigma)=\AUT X$ if, and only if $a\neq b(c-1)/(b-1)$.
Moreover, we have the following two cases:
\begin{enumerate}[(1)]
\item Suppose that $\AUT(X,\sigma)=\AUT X$:
\begin{enumerate}[(a)]
\item if we have $a=bc$ or $a=(b-c)/(c-1)$ or $a=1+c-c/b$, then $\AUT X\simeq D_2$;
\item if we have $bc=a=(b-c)/(c-1)$ or $bc=a=1+c-c/b$ or $(b-c)/(c-1)=a=1+c-c/b$, then $\AUT X\simeq D_6$;
\item else, $\AUT X\simeq\Z/2\Z$.
 \end{enumerate}
\item Suppose that $\AUT(X,\sigma)\neq\AUT X$:
\begin{enumerate}[(a)]
\item if we have $a=bc$ or $a=(b-c)/(c-1)$ or $a=1+c-c/b$, then $\AUT(X,\sigma)\simeq D_2$ and $\AUT X\simeq D_4$;
\item if we have $bc=a=(b-c)/(c-1)$ or $bc=a=1+c-c/b$ or $(b-c)/(c-1)=a=1+c-c/b$, then
\[
(a,b,c)=\left(\frac{1}{3},\frac{1}{2},\frac{2}{3}\right),
\]
$\AUT(X,\sigma)\simeq D_6$ and $\AUT X$ is isomorphic to the group $G_{24}$ of order $24$ admitting the following presentation:
\begin{equation}\label{eq:group_G24}
G_{24}=\langle r,s | r^4,s^6, (rs)^2, (r^{-1}s)^2\rangle;
\end{equation}
\item else, $\AUT(X,\sigma)\simeq\Z/2\Z$ and $\AUT X\simeq D_2$.
\end{enumerate}
\end{enumerate}
\end{theorem}

The details of the proof can be found in \cite[Theorem~4.2]{cirre01} and \cite[Proposition~3.5]{cirre03}.\\

This classification can also be translated in terms of the period matrices and of the hyperbolic structure of the algebraic curve.
See \cite[chap.~3]{rodriguez10} for more details, see also \cite{natanson78}.
%%%%%%%%%%%%%%%%%%%%%%%%%%%%%%
%
\section{Real M-curves with automorphisms in $W_9$}
\subsection{Description of the family $W_9$}\label{subsec:description_of_w9}
The $\SL_2(\R)$-orbit of the L-shaped translation surface tiled by three squares is denoted by $\Omega W_9$.
It projects into the Riemann moduli space $\mathcal{M}_2$ as an irreducible algebraic curve, which is denoted by $W_9$.
According to \cite[Remark, p.~3]{hubert&lelievre05}, the latter is the modular curve defined by the quotient $\lquotient{\Gamma}{\mathbb{H}}$ where $\Gamma$ is the level $2$ congruence subgroup generated by
\[
\begin{pmatrix}1&2\\0&1\end{pmatrix}
\text{ and }
\begin{pmatrix}0&-1\\1&0\end{pmatrix}.
\]
In \cite[\S4]{moeller05}, M.~M\"oller provides equations for the family $W_9$.
R.~Silhol gives in \cite[Theorem~A]{silhol07} a description of the algebraic curves in this family and shows that a translation surface $(X,\omega)$ is an element of $\Omega W_9$ if, and only if the curve $X$ admits an equation of the form
\begin{align}
y^2&=P_u(x)=x(x-1)\left(x^3+ux^2-\frac{8}{3}ux+\frac{16}{9}u\right)\label{eq:w9_magic_form}\\
\text{with }\omega&=\lambda\frac{xdx}{y}\nonumber
\end{align}
for $u\in\C\setminus\{-9,0\}$ and for some constant $\lambda\in\C^*$.
The degree $3$ polynomial
\[
x^3+ux^2-\frac83ux+\frac{16}{9}u
\]
admits exactly three real distinct roots if, and only if its discriminant is positive, that is, if, and only if $u$ is a real number such that $u<-9$.
Thus, a genus $2$ curve $X=X_u$ defined by an equation of the form \eqref{eq:w9_magic_form} is a real M-curve if, and only if $u<-9$.

The following fact was observed in \cite{silhol07} but not stated.

\begin{proposition}
There exists a bijection between the set of real automorphism classes of real M-curves in $W_9$ and the interval $[-18;-9[$.
\end{proposition}

\begin{proof}
In \cite[Proposition~4.1]{silhol07}, the author considers a certain family whose image in moduli space is shown to be the quotient of $\mathbb{P}^1(\mathbb{\C})\setminus\{-9,0,\infty\}$ under $u\mapsto -9u/(u+9)$ (see also \cite[Remarks~4.11]{silhol07}).
In particular, there is a bijection between this image and $\mathbb{P}^1(\mathbb{\C)}$ minus two points and a cone point of order $2$, namely the one corresponding to $u=-18$.
Is is then shown in \cite[Lemma~5.4]{silhol07} that the aforementioned family is $W_9$.

 If we now consider the case of real M-curves, then by the above discussions there is a bijection between $W_9^{\mathrm{M}}$ and $[-18;-9[$ on the one hand, and between $W_9^{\mathrm{M}}$ and $]-\infty; -18]$ on the other hand.
Furthermore, it results from \cite[Proposition~4.1 and Lemma~5.4]{silhol07} that if $X_u$ is a real M-curve defined by an equation $y^2=P_u(x)$ as in Equation \eqref{eq:w9_magic_form}, then the transformation $x\mapsto x/(x-1)$ induces a (complex, but non-real) isomorphism between $X_u$ and $X_{u'}$ with $u'=-9u/(u+9)$.
\end{proof}

Defining
\[
f_3:x\mapsto\frac{x+\sqrt{3}}{-\sqrt{3}x+1},
\]
we note that the M\"obius transformation $f_3$ is of order $3$ and fixes the points $i$ and $-i$.
The orbits of $0$ and $\infty$ under $f_3$ are respectively
\[
\{\sqrt3,-\sqrt3,0\}\text{ and }\{-\sqrt3/3,\sqrt3/3,\infty\}.
\]
We consider a genus $2$ algebraic curve defined by an equation of the form
\[
y^2=Q_s(x)=x(x+1)(x-s^2)\big(x-f_3(s)^2\big)\Big(x-f_3\big(f_3(s)\big)^2\Big).
\]
If we apply the transformation $x\mapsto x+1$ to the roots of $Q_s(x)$, then a direct calculation shows that this curve also admits an equation of the form \eqref{eq:w9_magic_form} with
\[
u=\frac{-81(s^2+1)^3}{(3s+\sqrt{3})^2(3s-\sqrt{3})^2}.
\]
We note that the function
\[
s\mapsto g(s)=\frac{-81(s^2+1)^3}{(3s+\sqrt3)^2(3s-\sqrt3)^2}
\]
is even and invariant under $f_3$.
This leads to the following:

\begin{proposition}\label{prop:description_of_w9}
Let $(X,\omega)$ be a translation surface in $\mathcal{H}(2)$.
Then $(X,\omega)$ is an element of $\Omega W_9$ if, and only if the algebraic curve $X$ admits an equation of the form
\begin{align}
y^2&=Q_s(x)=x(x+1)(x-s^2)\Big(x-f_3\big(f_3(s)\big)^2\Big)\big(x-f_3(s)^2\big)\label{eq:w9_order3_auto_form}\\
\text{with }\omega&=\lambda\left(\frac{dx}{y}+\frac{xdx}{y}\right)\nonumber
\end{align}
with $s\in\C\setminus\{-\sqrt{3},-\sqrt3/3,-1,0,\sqrt3/3,\sqrt{3},-i,i\}$ and for some constant $\lambda\in\C^*$.
\end{proposition}

\begin{proof}
According to \cite[Theorem~A]{silhol07}, a translation surface $(X,\omega)$ in $\mathcal{H}(2)$ is an element of $\Omega W_9$ if, and only if $X$ admits an equation of the form $y^2=P_u(x)$ as defined by \eqref{eq:w9_magic_form} with $u\in\C\setminus\{-9,0\}$ and $\omega$ admitting a double zero at $0$.

The condition in the statement is already proven to be sufficient by the discussion above.

Conversely, let $(X,\omega)$ be an element of $\Omega W_9$, then $X$ is defined by an equation of the form $y^2=P_u(x)$ with $u\in\C\setminus\{-9,0\}$. From the study of the function $g$ we can take $u=g(s)$ for some $s\in\C\setminus\{-\sqrt{3},-\sqrt3/{3},-1,0,\sqrt3/3,\sqrt{3},-i,i\}$.
Applying the transformation $x\mapsto x-1$ to the roots of the equation $y^2=P_u(x)$ then leads to the announced form \eqref{eq:w9_order3_auto_form} of the equation $y^2=Q_s(x)$.
Moreover, up to a non-zero, complex scalar multiple, $dx/y+xdx/y$ is the only holomorphic $1$-form on $X$ admitting a double zero at $(-1,0)$.
\end{proof}

Noting that we have $g(s)\leq-9$ for every $s\in\R$ and that
\[
\forall s\in]0;\sqrt3/3[,\quad s^2<f_3\big(f_3(s)\big)^2<f_3(s)^2,
\]
we obtain another description of this set.

\begin{lemma}\label{lemm:description_of_w9m}
Let $s$ be a real number such that $0<s<\sqrt3/3$, implying
\[
\frac{\sqrt3}{3}<\left|f_3\big(f_3(s)\big)\right|<\sqrt3<|f_3(s)|.
\]
Let $X$ be the real M-curve defined by the equation
\[
y^2=Q_s(x)=x(x+1)\big(x-a(s)\big)\big(x-b(s)\big)\big(x-c(s)\big),
\]
with $a(s)=s^2$, $b(s)=f_3\big(f_3(s)\big)^2$ and $c(s)=f_3(s)^2$.
Then $X$ is in the family $W_9$.
Conversely, every real M-curve in the family $W_9$ admits such a description.
\end{lemma}

As a consequence, we have the following:

\begin{proposition}
There exists a bijection between the set of real automorphism classes of real M-curves in $W_9$ and the interval $]0;\sqrt{3}/{3}[$.
\end{proposition}
%%%%%%%%%%
%
\subsection{Automorphisms}\label{subsec:proof_autos_w9m}
The preceding discussion encourages to consider the curve defined by an equation of the form \eqref{eq:w9_magic_form} with $u=-18$, that is
\[
y^2=P_u(x)=x(x-1)(x-8+4\sqrt3)(x-2)(x-8-4\sqrt3).
\]
This curve admits an order $4$ automorphism induced by the M\"obius transformation $x\mapsto x/(x-1)$.
Applying the transformation $x\mapsto x-1$ to the set of the roots of the polynomial $P_u$, we then get the equation
\[
y^2=x(x^2-1)(x-a)(x-1/a)\text{ with }a=7-4\sqrt3.
\]
By Example \ref{exam:3-square-tiled_surface} this curve is defined by the $3$-square-tiled translation surface.
We note that $7-4\sqrt3=(2-\sqrt3)^2$ with $0<2-\sqrt3<\sqrt3/3$, and that
\[
f_3(2-\sqrt3)^2=7+4\sqrt3=\frac1a\text{ and }f_3\big(f_3(2-\sqrt3)\big)^2=1,
\]
hence this curve is defined by an equation of the form \eqref{eq:w9_order3_auto_form} with $s=2-\sqrt3$.

By Example \ref{exam:order4_automorphism}, the Riemann surface defined by the $3$-square-tiled translation surface admits an order $4$ automorphism induced by an affine transformation.
In order to weaken this condition, we could consider lower order automorphisms of the Riemann surface, for example non-hyperelliptic holomorphic involutions.
However, we observe that in genus $2$ such automorphisms can not be induced by affine transformations:

\begin{lemma}
Let $(X,\omega)$ be a translation surface in $\mathcal{H}(2)$ and $\varphi$ be an order $2$ automorphism defined by an order $2$ automorphism affine with respect to $\omega$.
Then $\varphi$ is the hyperelliptic involution $h_X$.
\end{lemma}

\begin{proof}
Let $\varphi$ be such an automorphism: since $X$ is hyperelliptic, then following \cite[Theorem~V.2.13]{farkas&kra92}, $\varphi$ must satisfy one of the following:
\begin{enumerate}[(a)]
\item $\varphi$ has no fixed point;
\item $\varphi$ only fixes non-Weierstrass points;
\item $\varphi$ is the hyperelliptic involution.
\end{enumerate}
Now $\varphi$ must fix the double zero of $\omega$, which is necessarily a Weierstrass point, hence the conclusion.
\end{proof}

We now prove Proposition \ref{prop:autos_w9m} restated as follows.

\begin{proposition}
Let $s$ be a real number such that $0<s<\sqrt3/3$ and let $X$ be the real M-curve defined by the equation
\[
y^2=Q_s(x)=x(x+1)(x-s^2)\Big(x-f_3\big(f_3(s)\big)^2\Big)\big(x-f_3(s)^2\big)
\]
Then $\AUT X\neq\langle h_X\rangle$ if, and only if $s=2-\sqrt3$.
\end{proposition}

\begin{proof}
Considering Theorem \ref{theo:cirre}, it is sufficient to consider M\"obius transformations that could give an order $2$ automorphism.
By applying a suitable M\"obius transformation to the roots of $Q_s$, we can easily adapt the conditions ensuring the existence of non-trivial automorphisms stated in the theorem to equations of the form \eqref{eq:w9_order3_auto_form} with $0<a(s)<b(s)<c(s)$ for all $s\in]0;\sqrt3/3[$.
The possibilities are the following:
\begin{enumerate}[(a)]
\item the M\"obius transformation
\[
x\mapsto\frac{c(s)-b(s)}{b(s)-a(s)}\frac{x-a(s)}{x-c(s)},
\]
maps $(a,b,c)$ to $(0,-1,\infty)$ and induces a holomorphic involution on $X$ if, and only if we have
\[
a(s)=b(s)-1+\frac{b(s)}{c(s)}.
\]
Calculations lead to
\[
\frac{b(s)}{c(s)}=-\frac{16\sqrt3x(x^2+1)}{(x+\sqrt3)^2(\sqrt3x+1)^2}
\]
then to
\[
b(s)-1+\frac{b(s)}{c(s)}-a(s)=-\frac{(x^2+1)(3x^4+8\sqrt3x^3+18x^2+16\sqrt3x-9)}{(x+\sqrt3)^2(\sqrt3x+1)^2}.
\]
The degree $4$ polynomial in the numerator admits two real roots, namely $-2-\sqrt3$ and $2-\sqrt3$, the latter being the only one in the interval $]0;\sqrt3/3[$.
According to the above discussion at the beginning of this section, this curve has an order $4$ automorphism;
\item the M\"obius transformation
\[
x\mapsto a(s)\frac{x+1}{x-a(s)}
\]
induces an automorphism on $X$ if, and only if
\[
a(s)=\frac{b(s)c(s)}{1+b(s)+c(s)}.
\]
The only real solutions lead to the singular curve of equation
\[
y^2=x(x+1)(x-\sqrt3/3)(x+\sqrt3/3)
\]
corresponding to $s=\sqrt3/3$;
\item the M\"obius transformation
\[
x\mapsto\frac{c(s)-x}{x+1}
\]
induces an automorphism if, and only if
\[
a(s)=\frac{c(s)-b(s)}{1+b(s)},
\]
which gives the singular curve defined by the equation
\[
y^2=x^2(x+1)(x-\sqrt{3})(x+\sqrt{3})\]
corresponding to $s=0$.
Note that the M\"obius transformation $x\mapsto1/x$ induces an automorphism between the two curves defined by $s=0$ and $s=\sqrt3/3$;
\item lastly, the curve $X$ has a holomorphic involution induced by
\[
x\mapsto\frac{a(s)-b(s)}{b(s)-c(s)}\frac{x-c(s)}{x-a(s)}
\]
if, and only if
\[
a(s)=\frac{b(s)}{1-b(s)+c(s)},
\]
which also gives the solution $s=2-\sqrt{3}$.
\end{enumerate}
Calculations are tedious, but routine.
\end{proof}
%%%%%%%%%%%%%%%%%%%%%%%%%%%%%%
%
\section{Periods of real M-curves in $W_9$}\label{sec:periods_of_real_mcurves_in_w9}
\subsection{Theta characteristics}\label{subsec:theta-characteristics}
\subsubsection*{Definitions and elementary properties}\label{paragraph:definitions_and_properties}
Let $z\in\C^g$ and $Z\in\frak{S}_g$, the Riemann $\vartheta$ function is defined by
\begin{equation}\label{eq:theta}
\vartheta(z,Z)=\sum_{k\in\Z^g} \exp(\pi i\T{k}Zk+2\pi i\T{k}z).
\end{equation}
One shows that the above series defines a holomorphic function on $\C^g\times\frak{S}_g$ (see \cite[chap.~II, Proposition~1.1]{mumford_tata1}).

Let $m,n\in\frac12\Z^g$, we define order $2$ theta characteristics by
\begin{align*}
\thetabc{2m\\2n}(z,Z) &= \sum_{k\in\Z^g} \exp\big(\pi i\T{(k+m)}Z(k+m)+2\pi i\T{(k+m)}(z+n)\big)\\
&= \exp\big(\pi i\T{m}Zm+2\pi i\T{m}(z+n)\big)\vartheta(z+Zm+n,Z),
\end{align*}
which, by the above identity, are also holomorphic functions on $\C^g\times\frak{S}_g$.

The following proposition points out well known properties of the function $z\mapsto\thetasc{2m\\2n}(z,Z)$.

\begin{proposition}\label{prop:theta-properties}
Let $Z\in\frak{S}_g$, $m,n\in\frac12\Z^g$ and $p,q\in\Z^g$, then for every $z\in\C^g$, the functions $\thetasc{2m\\2n}$ satisfy:
\begin{equation}\label{eq:theta-quasi_periodicity}
\thetabc{2m\\ 2n}(z,Z) = \exp\big(\pi i\T{p}Zp+2\pi i\T{p}(z+n)-2\pi i \T{m}q\big) \thetabc{2m\\ 2n}(z+Zp+q,Z)
\end{equation}
and
\begin{equation}\label{eq:theta-mod2}
\thetabc{2m+2p\\ 2n+2q}(z,Z) = \exp(2\pi i\T{m}q) \thetabc{2m\\ 2n}(z,Z).
\end{equation}
One also has:
\begin{equation}\label{eq:theta-even_odd}
\thetabc{2m\\ 2n}(-z,Z) = \exp(4\pi i\T{m}n)\thetabc{2m\\ 2n}(z,Z).
\end{equation}
\end{proposition}

\noindent See for example \cite[p.~123]{mumford_tata1} for \eqref{eq:theta-quasi_periodicity} and \eqref{eq:theta-mod2}, and \cite[\S II.3, Proposition 3.14]{mumford_tata1} for \eqref{eq:theta-even_odd}.

\begin{remark}\label{rema:theta-characteristics_mod2}
According to identity \eqref{eq:theta-even_odd}, for all $m,n\in\frac12\Z^g$, the function $z\mapsto\thetasc{2m\\2n}(z,Z)$ is even (resp. odd) if and only if $4\T{m}n\equiv0\mod2$ (resp. $4\T{m}n\equiv1\mod2$).
There are exactly $2^{2g}$ order $2$ theta characteristics, among which $2^{g-1}(2^g+1)$ are even and $2^{g-1}(2^g-1)$ are odd functions (see \cite[Corollary~VI.1.5]{farkas&kra92}).
\end{remark}
%%%%%%%%%%
\subsubsection*{Modular transformation formula}\label{paragraph:modular_formula}
The following describes how $Z\mapsto\thetasc{2m\\2n}(z,Z)$ transforms under the action of the symplectic group.

\begin{theorem}\label{theo:theta_modular_formula}
Let $m,n\in\frac12\Z^g$ and $M=(\begin{smallmatrix}\alpha&\beta\\ \gamma&\delta\end{smallmatrix})\in\SP(2g,\Z)$.
For every $(z,Z)\in\C^g\times\frak{S}_g$ we define
\[
M(z,Z)=\left(\T{(\gamma Z+\delta)}^{-1}z,(\alpha Z+\beta)(\gamma Z+\delta)^{-1}\right).
\]
Then the following transformation formula holds:
\[\thetabc{2m'\\2n'}\big(M(z,Z)\big)=\zeta_M\exp\left(\pi i\T{z}(\gamma Z+\delta)^{-1}\gamma z\right)\det(\gamma Z+\delta)^{\frac{1}{2}}\thetabc{2m\\2n}(z,Z)
\]
with
\[
\begin{pmatrix}m'\\ n'\end{pmatrix}
=
\begin{pmatrix}\delta m-\gamma n\\ -\beta m+\alpha n\end{pmatrix}
+\frac12
\begin{pmatrix}\DIAG(\gamma\T{\delta})\\ \DIAG(\alpha\T{\beta})\end{pmatrix},
\]
where $\DIAG(N)=(N_{11},\ldots,N_{gg})$ for $N\in\M_g(\C)$, and where $\zeta_M\in\C^*$ is a eighth root of the unity only depending on $M$.
\end{theorem}

\noindent See \cite[\S II.5, pp.~189-197]{mumford_tata1} for a proof.
%%%%%%%%%
\subsubsection*{Theta characteristics and hyperelliptic curves}\label{paragraph:theta-characteristics_and_hyperelliptic}
Let $X$ be a genus $g$ curve and $Z\in\frak{S}_g$ a period matrix of $X$.
Let $m,n\in(\rquotient{\Z}{2\Z})^g$, from now on we will denote
\[
\thetabc{m\\ n}(Z)=\thetabc{m\\ n}(0,Z).
\]
The \emph{theta characteristics of $Z$} are the values $\thetasc{m\\ n}(Z)$.
A theta characteristic is said to be even (resp. odd) if $\T{m}n$ is even (resp. odd).
In particular, every odd theta characteristic is zero.

\begin{remark}\label{rema:theta-mod2}
If $\thetasc{m\\ n}(Z)$ is zero, then by \eqref{eq:theta-mod2}, for all $p,q\in(\Z/2\Z)^g$ such that $p,q\equiv0\mod2$, the theta characteristic $\thetasc{m+p\\ n+q}(Z)$ is also zero.
\end{remark}

The important role played by theta characteristics in the theory of hyperelliptic curves is illustrated by the following result, stated here in the specific case of genus $3$.

\begin{theorem}\label{theo:hyperelliptic_genus3_thetanull}
Let $Z$ be a period matrix of a genus $3$ curve $X$.
If $X$ is not hyperelliptic, then no even theta characteristic of $Z$ is zero, and $X$ is hyperelliptic if, and only if exactly one even theta characteristic of $Z$ is zero.
\end{theorem}

\noindent The reader may refer to \cite[\S IIIa.9, Theorem~9.1]{mumford_tata2}.
%%%%%%%%%%
%
\subsection{Double cover associated to a curve in $W_9$}\label{subsec:double_cover}
In this section we describe the construction of a non ramified double cover of a translation surface in $\mathcal{H}(2)$ tiled by three parallelograms.
%%%%%%%%%%
\subsubsection*{Construction}\label{paragraph:construction_of_double_covers}
Let $(X,\omega)$ be a translation surface in the $\SL_2(\R)$-orbit of the L-shaped surface tiled by three squares.
Such a surface is geometrically defined by the quotient $(\rquotient{\mathcal{P}}{\sim},dz)$ where $\mathcal{P}$ is the Euclidean hexagon obtained by assembling three copies of a parallelogram of unit area, as decribed in Figure \ref{fig:double_cover}.

\begin{figure}[!ht]
\centering \ifx\JPicScale\undefined\def\JPicScale{1}\fi
\unitlength \JPicScale mm
\begin{picture}(100.94,70)(0,0)
\put(17.5,35){\makebox(0,0)[cc]{$(\hat{X},\hat{\omega})$}}

\linethickness{0.3mm}
\put(0,0){\line(1,0){10}}
\linethickness{0.3mm}
\put(12.5,10){\line(1,0){10}}
\linethickness{0.3mm}
\put(5,20){\line(1,0){10}}
\linethickness{0.3mm}
\multiput(0,0)(0.12,0.48){42}{\line(0,1){0.48}}
\linethickness{0.3mm}
\multiput(10,0)(0.12,0.48){21}{\line(0,1){0.48}}
\linethickness{0.3mm}
\multiput(27.5,30)(0.12,0.48){21}{\line(0,1){0.48}}
\linethickness{0.3mm}
\put(17.5,30){\line(1,0){10}}
\linethickness{0.3mm}
\multiput(35,20)(0.12,0.48){42}{\line(0,1){0.48}}
\linethickness{0.3mm}
\put(30,40){\line(1,0){10}}
\linethickness{0.15mm}
\put(5,15){\line(1,0){17.5}}
\put(22.5,15){\vector(1,0){0.12}}
\put(5,15){\vector(-1,0){0.12}}
\linethickness{0.15mm}
\put(17.5,25){\line(1,0){17.5}}
\put(35,25){\vector(1,0){0.12}}
\put(17.5,25){\vector(-1,0){0.12}}
\linethickness{0.3mm}
\put(22.5,10){\circle*{1.88}}

\put(22.5,50){\makebox(0,0)[cc]{$(X,\omega)$}}

\linethickness{0.3mm}
\put(0,45){\line(1,0){10}}
\linethickness{0.3mm}
\put(12.5,55){\line(1,0){10}}
\linethickness{0.3mm}
\put(5,65){\line(1,0){20}}
\linethickness{0.3mm}
\multiput(0,45)(0.12,0.48){42}{\line(0,1){0.48}}
\linethickness{0.3mm}
\multiput(10,45)(0.12,0.48){21}{\line(0,1){0.48}}
\linethickness{0.3mm}
\multiput(22.5,55)(0.12,0.48){21}{\line(0,1){0.48}}
\linethickness{0.3mm}
\put(60,45){\line(1,0){10}}
\linethickness{0.3mm}
\put(72.5,55){\line(1,0){10}}
\linethickness{0.3mm}
\put(65,65){\line(1,0){20}}
\linethickness{0.3mm}
\multiput(60,45)(0.12,0.48){42}{\line(0,1){0.48}}
\linethickness{0.3mm}
\multiput(70,45)(0.12,0.48){21}{\line(0,1){0.48}}
\linethickness{0.3mm}
\multiput(82.5,55)(0.12,0.48){21}{\line(0,1){0.48}}
\put(80,70){\makebox(0,0)[cc]{$P_1$}}

\put(55,50){\makebox(0,0)[cc]{$P_2$}}

\put(58.75,65){\makebox(0,0)[cc]{$\mathcal{P}$}}

\linethickness{0.3mm}
\put(60,0){\line(1,0){10}}
\linethickness{0.3mm}
\put(72.5,10){\line(1,0){10}}
\linethickness{0.3mm}
\put(65,20){\line(1,0){10}}
\linethickness{0.3mm}
\multiput(60,0)(0.12,0.48){42}{\line(0,1){0.48}}
\linethickness{0.3mm}
\multiput(70,0)(0.12,0.48){21}{\line(0,1){0.48}}
\linethickness{0.3mm}
\multiput(82.5,10)(0.12,0.48){21}{\line(0,1){0.48}}
\linethickness{0.3mm}
\put(77.5,30){\line(1,0){10}}
\linethickness{0.3mm}
\multiput(95,20)(0.12,0.48){42}{\line(0,1){0.48}}
\linethickness{0.3mm}
\put(90,40){\line(1,0){10}}
\linethickness{0.3mm}
\put(10,0){\circle*{1.88}}

\linethickness{0.3mm}
\put(2.5,10){\circle*{1.88}}

\linethickness{0.3mm}
\put(15,20){\circle*{1.88}}

\linethickness{0.3mm}
\put(27.5,30){\circle*{1.88}}

\linethickness{0.3mm}
\put(35,20){\circle*{1.88}}

\linethickness{0.3mm}
\put(40,40){\circle*{1.88}}

\linethickness{0.3mm}
\put(10,45){\circle*{1.88}}

\linethickness{0.3mm}
\put(25,65){\circle*{1.88}}

\linethickness{0.3mm}
\put(22.5,55){\circle*{1.88}}

\linethickness{0.3mm}
\put(12.5,55){\circle*{1.88}}

\linethickness{0.3mm}
\put(2.5,55){\circle*{1.88}}

\linethickness{0.3mm}
\put(0,45){\circle*{1.88}}

\linethickness{0.3mm}
\put(15,65){\circle*{1.88}}

\linethickness{0.3mm}
\put(5,65){\circle*{1.88}}

\linethickness{0.3mm}
\put(60,45){\circle*{1.88}}

\linethickness{0.3mm}
\put(75,65){\circle*{1.88}}

\linethickness{0.3mm}
\put(65,65){\circle*{1.88}}

\linethickness{0.3mm}
\put(62.5,55){\circle*{1.88}}

\linethickness{0.3mm}
\put(72.5,55){\circle*{1.88}}

\linethickness{0.3mm}
\put(70,45){\circle*{1.88}}

\linethickness{0.3mm}
\put(85,65){\circle*{1.88}}

\linethickness{0.3mm}
\put(82.5,55){\circle*{1.88}}

\linethickness{0.3mm}
\put(62.5,10){\circle*{1.88}}

\linethickness{0.3mm}
\put(70,0){\circle*{1.88}}

\linethickness{0.3mm}
\put(75,20){\circle*{1.88}}

\linethickness{0.3mm}
\put(82.5,10){\circle*{1.88}}

\linethickness{0.3mm}
\put(87.5,30){\circle*{1.88}}

\linethickness{0.3mm}
\put(95,20){\circle*{1.88}}

\linethickness{0.3mm}
\put(100,40){\circle*{1.88}}

\linethickness{0.3mm}
\put(68.75,15){\circle{1.88}}

\linethickness{0.3mm}
\multiput(0,6.25)(0.12,-0.12){21}{\line(1,0){0.12}}
\linethickness{0.3mm}
\multiput(10,6.25)(0.12,-0.12){21}{\line(1,0){0.12}}
\linethickness{0.3mm}
\multiput(10,6.88)(0.12,-0.12){21}{\line(1,0){0.12}}
\linethickness{0.3mm}
\multiput(37.5,36.25)(0.12,-0.12){21}{\line(1,0){0.12}}
\linethickness{0.3mm}
\multiput(27.5,35.62)(0.12,-0.12){21}{\line(1,0){0.12}}
\linethickness{0.3mm}
\multiput(27.5,36.25)(0.12,-0.12){21}{\line(1,0){0.12}}
\linethickness{0.3mm}
\multiput(3.75,43.75)(0.12,0.12){21}{\line(1,0){0.12}}
\linethickness{0.3mm}
\multiput(8.75,63.75)(0.12,0.12){21}{\line(1,0){0.12}}
\linethickness{0.3mm}
\multiput(18.75,63.75)(0.12,0.12){21}{\line(1,0){0.12}}
\linethickness{0.3mm}
\multiput(16.25,53.75)(0.12,0.12){21}{\line(1,0){0.12}}
\linethickness{0.3mm}
\multiput(15.62,53.75)(0.12,0.12){21}{\line(1,0){0.12}}
\linethickness{0.3mm}
\multiput(19.38,63.75)(0.12,0.12){21}{\line(1,0){0.12}}
\linethickness{0.3mm}
\multiput(2.5,61.25)(0.12,-0.12){21}{\line(1,0){0.12}}
\linethickness{0.3mm}
\multiput(22.5,61.25)(0.12,-0.12){21}{\line(1,0){0.12}}
\linethickness{0.3mm}
\multiput(10,51.25)(0.12,-0.12){21}{\line(1,0){0.12}}
\linethickness{0.3mm}
\multiput(0,51.25)(0.12,-0.12){21}{\line(1,0){0.12}}
\linethickness{0.3mm}
\multiput(0,50.62)(0.12,-0.12){21}{\line(1,0){0.12}}
\linethickness{0.3mm}
\multiput(10,51.88)(0.12,-0.12){21}{\line(1,0){0.12}}
\linethickness{0.15mm}
\multiput(30,21.25)(0.12,0.42){42}{\line(0,1){0.42}}
\put(35,38.75){\vector(1,4){0.12}}
\put(30,21.25){\vector(-1,-4){0.12}}
\linethickness{0.15mm}
\multiput(17.5,11.25)(0.12,0.42){42}{\line(0,1){0.42}}
\put(22.5,28.75){\vector(1,4){0.12}}
\put(17.5,11.25){\vector(-1,-4){0.12}}
\linethickness{0.15mm}
\multiput(5,1.25)(0.12,0.42){42}{\line(0,1){0.42}}
\put(10,18.75){\vector(1,4){0.12}}
\put(5,1.25){\vector(-1,-4){0.12}}
\linethickness{0.1mm}
\put(2.5,10){\line(1,0){10}}
\linethickness{0.3mm}
\multiput(15,20)(0.12,0.48){21}{\line(0,1){0.48}}
\linethickness{0.1mm}
\multiput(12.5,10)(0.12,0.48){21}{\line(0,1){0.48}}
\linethickness{0.3mm}
\put(25,20){\line(1,0){10}}
\linethickness{0.1mm}
\put(15,20){\line(1,0){10}}
\linethickness{0.1mm}
\put(27.5,30){\line(1,0){10}}
\linethickness{0.3mm}
\multiput(22.5,10)(0.12,0.48){21}{\line(0,1){0.48}}
\linethickness{0.1mm}
\multiput(25,20)(0.12,0.48){21}{\line(0,1){0.48}}
\linethickness{0.1mm}
\put(62.5,10){\line(1,0){10}}
\linethickness{0.3mm}
\put(85,20){\line(1,0){10}}
\linethickness{0.1mm}
\put(75,20){\line(1,0){10}}
\linethickness{0.3mm}
\multiput(87.5,30)(0.12,0.48){21}{\line(0,1){0.48}}
\linethickness{0.1mm}
\multiput(85,20)(0.12,0.48){21}{\line(0,1){0.48}}
\linethickness{0.1mm}
\put(87.5,30){\line(1,0){10}}
\linethickness{0.3mm}
\multiput(75,20)(0.12,0.48){21}{\line(0,1){0.48}}
\linethickness{0.1mm}
\multiput(72.5,10)(0.12,0.48){21}{\line(0,1){0.48}}
\linethickness{0.1mm}
\put(62.5,55){\line(1,0){10}}
\linethickness{0.1mm}
\multiput(72.5,55)(0.12,0.48){21}{\line(0,1){0.48}}
\linethickness{0.1mm}
\multiput(12.5,55)(0.12,0.48){21}{\line(0,1){0.48}}
\linethickness{0.1mm}
\put(2.5,55){\line(1,0){10}}
\linethickness{0.3mm}
\put(29.06,39.06){\rule{1.88\unitlength}{1.87\unitlength}}
\linethickness{0.3mm}
\put(36.56,29.06){\rule{1.88\unitlength}{1.87\unitlength}}
\linethickness{0.3mm}
\put(16.56,29.06){\rule{1.88\unitlength}{1.87\unitlength}}
\linethickness{0.3mm}
\put(24.06,19.06){\rule{1.88\unitlength}{1.87\unitlength}}
\linethickness{0.3mm}
\put(4.06,19.06){\rule{1.88\unitlength}{1.87\unitlength}}
\linethickness{0.3mm}
\put(11.56,9.06){\rule{1.88\unitlength}{1.87\unitlength}}
\linethickness{0.3mm}
\put(-0.94,-0.94){\rule{1.88\unitlength}{1.87\unitlength}}
\linethickness{0.3mm}
\put(71.56,9.06){\rule{1.88\unitlength}{1.87\unitlength}}
\linethickness{0.3mm}
\put(66.25,5){\circle{1.88}}

\linethickness{0.3mm}
\put(78.75,15){\circle{1.88}}

\linethickness{0.3mm}
\put(81.25,25){\circle{1.88}}

\linethickness{0.3mm}
\put(91.25,25){\circle{1.88}}

\linethickness{0.3mm}
\put(93.75,35){\circle{1.88}}

\linethickness{0.3mm}
\put(59.06,-0.94){\rule{1.88\unitlength}{1.87\unitlength}}
\linethickness{0.3mm}
\put(64.06,19.06){\rule{1.88\unitlength}{1.87\unitlength}}
\linethickness{0.3mm}
\put(84.06,19.06){\rule{1.88\unitlength}{1.87\unitlength}}
\linethickness{0.3mm}
\put(76.56,29.06){\rule{1.88\unitlength}{1.87\unitlength}}
\linethickness{0.3mm}
\put(96.56,29.06){\rule{1.88\unitlength}{1.87\unitlength}}
\linethickness{0.3mm}
\put(89.06,39.06){\rule{1.88\unitlength}{1.87\unitlength}}
\linethickness{0.3mm}
\put(66.25,50){\circle{1.88}}

\linethickness{0.3mm}
\put(61.25,50){\circle{1.88}}

\linethickness{0.3mm}
\put(71.25,50){\circle{1.88}}

\linethickness{0.3mm}
\put(77.5,55){\circle{1.88}}

\linethickness{0.3mm}
\put(78.75,60){\circle{1.88}}

\linethickness{0.3mm}
\put(80,65){\circle{1.88}}

\linethickness{0.3mm}
\put(68.75,60){\circle{1.88}}

\linethickness{0.3mm}
\multiput(88.12,35.62)(0.12,-0.12){10}{\line(1,0){0.12}}
\linethickness{0.3mm}
\multiput(88.12,34.38)(0.12,0.12){10}{\line(1,0){0.12}}
\linethickness{0.3mm}
\multiput(98.12,35.62)(0.12,-0.12){10}{\line(1,0){0.12}}
\linethickness{0.3mm}
\multiput(98.12,34.38)(0.12,0.12){10}{\line(1,0){0.12}}
\linethickness{0.3mm}
\multiput(79.38,20.62)(0.12,-0.12){10}{\line(1,0){0.12}}
\linethickness{0.3mm}
\multiput(79.38,19.38)(0.12,0.12){10}{\line(1,0){0.12}}
\linethickness{0.3mm}
\multiput(81.88,30.62)(0.12,-0.12){10}{\line(1,0){0.12}}
\linethickness{0.3mm}
\multiput(81.88,29.38)(0.12,0.12){10}{\line(1,0){0.12}}
\linethickness{0.3mm}
\multiput(76.88,10.62)(0.12,-0.12){10}{\line(1,0){0.12}}
\linethickness{0.3mm}
\multiput(76.88,9.38)(0.12,0.12){10}{\line(1,0){0.12}}
\linethickness{0.3mm}
\multiput(70.62,5.62)(0.12,-0.12){10}{\line(1,0){0.12}}
\linethickness{0.3mm}
\multiput(70.62,4.38)(0.12,0.12){10}{\line(1,0){0.12}}
\linethickness{0.3mm}
\multiput(60.62,5.62)(0.12,-0.12){10}{\line(1,0){0.12}}
\linethickness{0.3mm}
\multiput(60.62,4.38)(0.12,0.12){10}{\line(1,0){0.12}}
\end{picture}
\caption{Double cover: identifications and Weierstrass points}\label{fig:double_cover}
\end{figure}
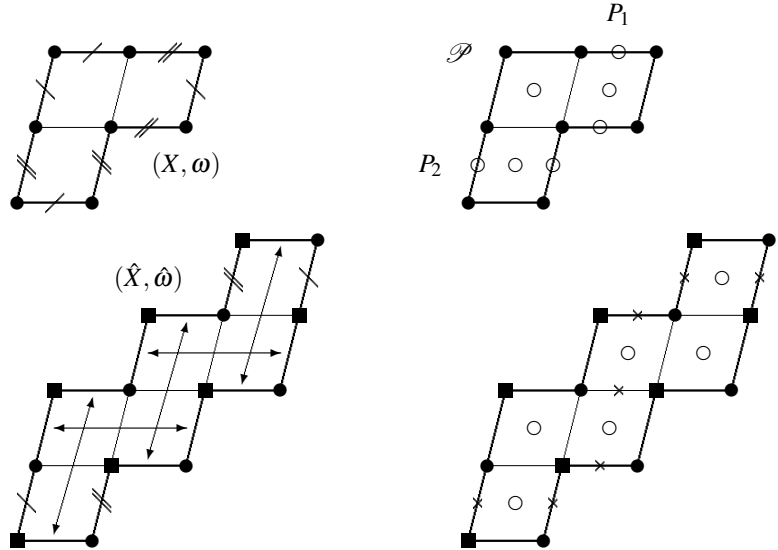

In this representation, the six Weierstrass points of $X$ correspond to:
\begin{itemize}
\item the cone-type singularity of angle $6\pi$, represented by black dots in Figure \ref{fig:double_cover};
\item the center of the parallelograms and the middle points of the two horizontal and vertical sides, pairwise identified, of the two non-adjacent parallelograms, all represented by small circles in Figure \ref{fig:double_cover}.
\end{itemize}

We construct the non-ramified double cover $p:(\hat{X},\hat{\omega})\to(X,\omega)$ by assembling $(X,\omega)$ with its image by the symmetry whose center is one of the two Weierstrass points that is neither the singularity, nor the center of a parallelogram.
More precisely, we can suppose that, up to a translation, the point denoted by $P_1$ in Figure \ref{fig:double_cover} is $0$ and we denote by $-\mathcal{P}$ the image of the polygon $\mathcal{P}$ by $z\mapsto-z$.
Note that the choice of one or another of the two Weierstrass points $P_1$ or $P_2$ has no incidence on the construction, as is readily checked by reassembling the parallelograms.
Then $\mathcal{P}\cup(-\mathcal{P})$ is an Euclidean dodecagon and we consider the identifications by translation specified in Figure \ref{fig:double_cover}.
The quotient $\big(\rquotient{\mathcal{P}\cup(-\mathcal{P})}{\sim},dz\big)$ defines a staircase-shaped translation surface tiled by six squares, whose vertices are identified to two cone-type singularities of angle $6\pi$, marked by black disks and squares in Figure \ref{fig:double_cover}.
We thus obtain a genus $3$ translation surface $(\hat{X},\hat{\omega})$ such that $\hat{\omega}$ has two double zeros on $\hat{X}$.
%%%%%%%%%%
\subsubsection*{Automorphisms}\label{subs:autos-revt-3c}
By a classical result of algebraic geometry, the algebraic curve $\hat{X}$ defined above is hyperelliptic, as a non-ramified cover of a genus $2$ curve; its eight Weierstrass points are
\begin{itemize}
\item the two cone-type singularities;
\item the centers of the parallelograms, represented by small circles in Figure \ref{fig:double_cover}.
\end{itemize}

The central symmetry defines an involution that fixes exactly four points. These points are represented by small crosses in Figure \ref{fig:double_cover} and correspond to:
\begin{itemize}
\item the center of symmetry of the assembling of parallelograms;
\item the middle points of the two horizontal sides of the two central parallelograms;
\item the middle points of the four vertical sides of the two parallelograms located at the extremities.
\end{itemize}

\noindent The central symmetry then induces a non-hyperelliptic order $2$ automorphism on $\hat{X}$, which will be denoted by $\psi_2$.
The curve $\hat{X}$ then admits an equation of the form 
\begin{align*}
w^2&=\hat{P}(z)=(z^2+1)(z^2-a^2)(z^2-b^2)(z^2-c^2)\\
\text{with }\hat{\omega}&=\frac{dz}{w}+\frac{z^2dz}{w},
\end{align*}
for which $\psi_2$ is defined by $(z,w)\mapsto(-z,w)$.

Composing $\psi_2$ with the hyperelliptic involution $h_{\hat{X}}:(z,w)\mapsto(z,-w)$ yields an extra involution $\tau:(z,w)\mapsto(-z,-w)$, that is fixed point free and such that $X=\hat{X}/\langle\tau\rangle$; the covering map $p:\hat{X}\to X$ is then given by $(z,w)\mapsto(z^2,zw)$.
The algebraic curve $X$ then admits the equation
\begin{align*}
y^2&=P(x)=x(x+1)(x-a^2)(x-b^2)(x-c^2)\\
\text{with }\omega&=\frac12\left(\frac{dx}{y}+\frac{xdx}{y}\right).
\end{align*}

Furthermore, the curve $\hat{X}$ admits an order $3$ automorphism defined as follows.
Consider the affine transformation induced by rotation of angle $\pi$ around each of the two cone-type singularities, operating by circular permutation of the three parallelograms belonging to a same diagonal row.
Letting this transformation act three times is equivalent to replace $\hat{\omega}$ by $-\hat{\omega}$.
Now this operation leaves invariant the quadratic differential $\hat{\omega}^2$ that induces the complex structure on $\hat{X}$ (see Remark \ref{rema:half-translation_surface}).
The automorphism thus defined on $\hat{X}$ is of order $3$ and will be denoted by $\psi_3$.
%%%%%%%%%%
\subsubsection*{Homology basis and periods}\label{paragraph:curve_and_cover_homology_bases}
Let $X$ be a genus $2$ algebraic curve defined by
\[
y^2=P(x)=x(x+1)(x-a^2)(x-b^2)(x-c^2).
\]
Let $\pi:X\to\mathbb{P}_{\C}^1$ be the projection $(x,y)\mapsto x$.
Using the construction presented in p.~\pageref{paragraph:hyperelliptic_homology_basis} applied to
\[
x_1=-1,\ x_2=0,\ x_3=a^2,\ x_4=b^2\ \text{and}\ x_ 5=c^2,
\]
we obtain simple closed curves $\delta_1,\ldots,\delta_6$ in $X$, oriented such that the intersection numbers are $(\delta_j\cdot\delta_{j+1})=1$ ($j\mod6$), all others being zero, and such that $\sum_{j=1}^3\delta_{2j}=0$ and $\sum_{j=1}^3\delta_{2j-1}=0$.

Let $\hat{X}$ be the genus $3$ hyperelliptic curve defined by
\begin{equation}\label{eq:double_cover}
w^2=\hat{P}(z)=(z^2+1)(z^2-a^2)(z^2-b^2)(z^2-c^2)
\end{equation}
and let $\hat{\pi}:\hat{X}\to\mathbb{P}_{\C}^1$ be the projection $(z,w)\mapsto z$.

We proceed the same way by applying the construction in p.~\pageref{paragraph:hyperelliptic_homology_basis} to
\[
x_1=-c,\ x_2=-b,\ x_3=-a,\ x_4=i,\ x_ 5=-i,\ x_6=a,\ x_7=b\ \text{and}\ x_8=c.
\]
We denote by $\gamma_1,\ldots,\gamma_8$ the obtained cycles, satisfying $(\gamma_j\cdot\gamma_{j+1})=1$ ($j\mod8$), all others being zero, such that $\sum_{j=1}^4\gamma_{2j}=0$ and $\sum_{j=1}^4\gamma_{2j-1}=0$.

\begin{figure}[!ht]
\centering \input{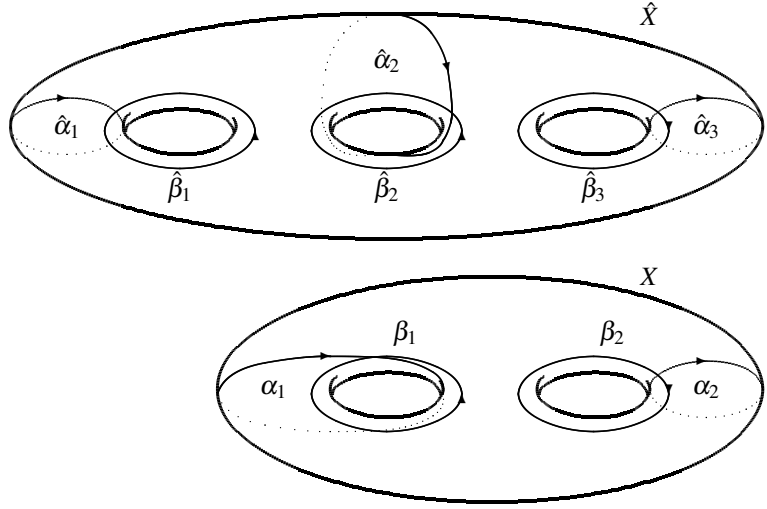}
\caption{Homology bases for $X$ and $\hat{X}$}\label{fig:homology_bases}
\end{figure}

Now let
\begin{equation}\label{eq:double_cover_homology_basis}
\begin{array}{ll}
\hat{\alpha}_1=\gamma_1, &\hat{\beta}_1=-\gamma_2,\\
\hat{\alpha}_2=\gamma_1+\gamma_3+\gamma_4, &\hat{\beta}_2=-\gamma_4,\\
\hat{\alpha}_3=\gamma_7, &\hat{\beta}_3=\gamma_6
\end{array}
\end{equation}
and
\begin{equation}\label{eq:curve_homology_basis}
\begin{array}{ll}
\alpha_1=\delta_1-\delta_6, &\beta_1=\delta_1,\\
\alpha_2=\delta_4, &\beta_2=\delta_3.
\end{array}
\end{equation}
Then $\hat{\mathcal{B}}=(\hat{\alpha}_1,\hat{\alpha}_2,\hat{\alpha}_3,\hat{\beta}_1,\hat{\beta}_2,\hat{\beta}_3)$ and $\mathcal{B}=(\alpha_1,\alpha_2,\beta_1,\beta_2)$ are symplectic bases for $\HH_1(\hat{X},\Z)$ and $\HH_1(X,\Z)$ respectively, represented in Figure \ref{fig:homology_bases}.

\begin{remark}\label{rema:hyperelliptic_M-curve_homology_basis}
If $a^2$, $b^2$ and $c^2$ are positive real numbers, then $X$ is a real genus $2$ M-curve and $\hat{X}$ is a real genus $3$ curve with three real components: this case corresponds to the situation where the translation surfaces $(X,\omega)$ and $(\hat{X},\hat{\omega})$ are tiled by rectangles.
We can then specialize the above construction: we call $\delta_1$ the pullback of $[-1;0]$ to $X$ by $\pi$, $\delta_2$ the pullback of $[0;a^2]$ and so on.
Since $P$ is non-zero in the upper half-plane $\mathbb{H}$ that is simply connected, we can choose on $\mathbb{H}$ the determination of $\sqrt{P}$ that is real and positive on $[-1;0]$.
This determination can be extended to $\R$ and even to the strips below the open intervals bounded by the roots of $P(x)$.
Note that it will then be pure imaginary with positive imaginary part on $]-\infty;-1]$, pure imaginary with negative imaginary part on $[0;a]$, real and negative on $[a;b]$, and so on (see also \cite[Lemma~2.4]{silhol01b}).

We construct in the same way the cycles $\gamma_j$'s on $\hat{X}$ for $j=1,\ldots,8$: the paths $\varepsilon_i$ are chosen as intervals in $\R\cup\{\infty\}$ for $i\notin\{3,4,5\}$, $\varepsilon_4$ as the interval from $i$ to $-i$ contained in the imaginary axis and $\varepsilon_3$ and $\varepsilon_5$ symmetric to the origin.
The orientation is defined by the choice of a determination of $\sqrt{\hat{P}}$ that is real and positive on $[-b;-a]$.
\end{remark}

The above construction allows to give an expression of the periods of $X$ in terms of the periods of its double cover $\hat{X}$.

\begin{lemma}\label{lemm:double_cover_periods_auto2}
Let $X$ be a genus $2$ algebraic curve defined by an equation of the form
\[
y^2=P(x)=x(x+1)(x-a^2)(x-b^2)(x-c^2)
\]
where $a$, $b$ and $c$ are complex numbers such that $a^2$, $b^2$ and $c^2$ are distinct and different from $0$ and $-1$, and $\hat{X}$ the double cover defined by
\[
w^2=\hat{P}(z)=(z^2+1)(z^2-a^2)(z^2-b^2)(z^2-c^2).
\]
Then the period matrix $\hat{Z}$ of $\hat{X}$ associated to the basis $\hat{\mathcal{B}}$ of $\HH_1(\hat{X},\Z)$ has the form
\[
\hat{Z}=
\begin{pmatrix}
z_1&z_{12}&z_{13}\\
z_{12}&z_2&z_{12}\\
z_{13}&z_{12}&z_1
\end{pmatrix}
\]
and the period matrix $Z$ of $X$ associated to the basis $\mathcal{B}$ of $\HH_1(X,\Z)$ is given by
\[
Z=
\begin{pmatrix}
2z_{2}&2z_{12}\\
2z_{12}&z_1+z_{13}
\end{pmatrix}.
\]
\end{lemma}
 
\begin{proof}
In order to make the calculation clear, the reader can consider the specialization of the construction of the homology basis $\hat{\mathcal{B}}$ to the case where $a$, $b$ and $c$ are real numbers such that $0<a<b<c$, see Remark \ref{rema:hyperelliptic_M-curve_homology_basis} above.

The curve $\hat{X}$ admits the non-hyperelliptic order $2$ automorphism defined by
\[
\psi_2:(z,w)\mapsto(-z,w).
\]
We note that ${\psi_2}_*(\gamma_1)=-\gamma_7$ as $\gamma_1$ and $\gamma_7$ have the same orientation, and that ${\psi_2}_*(\gamma_2)=\gamma_6$ since $\gamma_2$ and $\gamma_6$ have opposite orientations.
Moreover, we have ${\psi_2}_*(\gamma_3)=-\gamma_5$ and ${\psi_2}_*(\gamma_4)=-\gamma_4$.
That yields to ${\psi_2}_*(\hat{\alpha}_1)=-\hat{\alpha}_3$, ${\psi_2}_*(\hat{\alpha}_2)=-\hat{\alpha}_2$ and ${\psi_2}_*(\hat{\alpha}_3)=-\hat{\alpha}_1$ and the same holds for the $\hat{\beta}_j$'s.
We thus obtain the rational representation of $\psi_2$ in the symplectic basis $\hat{\mathcal{B}}$:
\[
\T{M_2}=
\begin{pmatrix}
0&0&-1&0&0&0\\
0&-1&0&0&0&0\\
-1&0&0&0&0&0\\
0&0&0&0&0&-1\\
0&0&0&0&-1&0\\
0&0&0&-1&0&0
\end{pmatrix}
=\rho_\Q(\psi_2).
\]
Then the matrix $\hat{Z}$ is stable under the action of the symplectic matrix $M_2$, hence the expression for $\hat{Z}$.
We denote by
\[
(\hat{\omega}_1,\hat{\omega}_2,\hat{\omega}_3)=\left(\frac{dz}{w},\frac{zdz}{w},\frac{z^2dz}{w}\right)
\]
the usual basis for the holomorphic $1$-forms on $\hat{X}$, and we introduce the matrices
\[
\hat{A}=\left(\int_{\hat{\alpha}_j}\hat{\omega}_k\right)_{j,k}\text{ and }\hat{B}=\left(\int_{\hat{\beta}_j}\hat{\omega}_k\right)_{j,k}.
\]
Noting that ${\psi_2}^*(\hat{\omega}_k)=-\hat{\omega}_k$ for $k=1$ and $3$, we deduce from the above calculation of $M_2$ that $\hat{A}_{11}=\hat{A}_{31}$ and $\hat{A}_{13}=\hat{A}_{33}$.
Also, since ${\psi_2}^*(\hat{\omega}_2)=\hat{\omega}_2$ then we have $\hat{A}_{12}=-\hat{A}_{32}$ and $\hat{A}_{22}=0$.
We obtain identical relations for $\hat{B}$, thus we have:
\[
\hat{A}=
\begin{pmatrix}
a_1&a_2&a_3\\
a_{21}&0&a_{23}\\
a_1&-a_2&a_3
\end{pmatrix}
\text{ and }
\hat{B}=\begin{pmatrix}
b_1&b_2&b_3\\
b_{21}&0&b_{23}\\
b_1&-b_2&b_3
\end{pmatrix}.
\]
Then we get $\hat{Z}=\hat{A}\hat{B}^{-1}$ under the announced form, with
\begin{align*}
z_1&=\frac{1}{2}\frac{a_2}{b_2}+\frac{1}{2}\frac{a_1b_{23}-a_3b_{21}}{b_1b_{23}-a_3b_{21}},\\
z_{12}&=\frac{a_3b_1-a_1b_3}{b_1b_{23}-b_3b_{21}}=\frac{1}{2}\frac{a_{21}b_{23}-a_{23}b_{21}}{b_1b_{23}-b_3b_{21}},\\
z_{13}&=\frac{1}{2}\frac{a_1b_{23}-a_3b_{21}}{b_1b_{23}-b_3b_{21}}-\frac{1}{2}\frac{a_2}{b_2},\\
z_2&=\frac{a_{23}b_1-a_{21}b_3}{b_1b_{23}-b_3b_{21}}.
\end{align*}

The curve $X$ is the quotient of $\hat{X}$ by the fixed point free involution $\tau:(z,w)\mapsto(-z,-w)$, which leaves invariant $\hat{\omega}_1$ and $\hat{\omega}_3$ on $\hat{X}$.
These differentials then induce holomorphic forms on $X$, which will be denoted by $\omega_1$ and $\omega_2$ respectively.
Let
\[
A=\left(\int_{\alpha_j}\omega_k\right)_{j,k}\text{ and }B=\left(\int_{\beta_j}\omega_k\right)_{j,k}.
\]
Let $p:\hat{X}\to X$ be the quotient map, then $p$ sends $\hat{\alpha}_1$ and $\hat{\alpha}_3$ onto $\alpha_2$, and $\hat{\alpha}_2$ onto $\alpha_1$, whereas $\hat{\beta}_1$ and $\hat{\beta}_3$ are sent onto $\beta_2$, and $\hat{\beta}_2$ onto $\beta_1$.
Furthermore, the restriction of $p$ is of degree $2$ on $\hat{\beta}_2$ and bijective on each $\hat{\alpha}_j$, $\hat{\beta}_j$ for $j\neq 2$.
We then have
\[
A=
\begin{pmatrix}
a_{21}&a_{23}\\
a_1&a_3
\end{pmatrix}
\text{ and }
B=
\begin{pmatrix}
\frac12b_{21}&\frac12b_{23}\\
b_1&b_3
\end{pmatrix},
\]
and a direct calculation yields the expression of the period matrix $Z=AB^{-1}$ of the curve $X$.
\end{proof}
%%%%%%%%%%
%
\subsection{Characterization of the cover in terms of the periods}\label{subsec:characterization_of_cover}
This section is devoted to the proof of Theorem \ref{theo:double_cover_thetanull}.
In the following, we deal with real M-curves in $W_9$.
This is only to fix ideas because the arguments used here work verbatim for the non-real case: the only difficulty lies in expressing the choice of arcs in $\mathbb{P}^1(\mathbb{C})$ whose pullbacks will provide suitable cycles for the construction of the symplectic bases $\hat{\mathcal{B}}$ and $\mathcal{B}$.
The method employed to exhibit the form of the period matrix and identify the corresponding theta characteristic is inspired from \cite[Theorem~5.5]{silhol01b}.

We first establish the expression of a period matrix of $\hat{X}$ when $X$ is a real M-curve in $W_9$.

\begin{proposition}\label{prop:double_cover_periods_auto3}
Let $X$ be a real M-curve in the family $W_9$.
Let $\hat{Z}$ be the period matrix of its double cover $\hat{X}$ associated to the symplectic basis $\hat{\mathcal{B}}$ of $\HH_1(\hat{X},\Z)$ described above.
Then $\hat{Z}$ has the form
\begin{equation} \label{eq:genus3_auto3_periods}
\hat{Z}=
\begin{pmatrix}
iy_1 & \frac12iy_1 & iy_{13}\\
\frac12iy_1 & \frac12+\frac34iy_1-\frac12iy_{13} & \frac12iy_1\\
iy_{13} & \frac12iy_1 & iy_1
\end{pmatrix}
\end{equation}
with $y_1,y_{13}\in\mathbb{R}$.
\end{proposition}

\begin{proof}
Let $f_3$ be the M\"obius transformation defined in \S\ref{subsec:description_of_w9}, then $X$ admits the equation
\[
y^2=x(x+1)(x-a^2)(x-b^2)(x-c^2)
\]
where $a$, $b$ and $c$ are real numbers satisfying the conditions of Lemma \ref{lemm:description_of_w9m}.
The double cover $\hat{X}$ is defined by
\[
w^2=(z^2+1)(z^2-a^2)(z^2-b^2)(z^2-c^2),
\]
and $f_3$ induces an order $3$ automorphism on $\hat{X}$, denoted by $\psi_3$, defined by
\[
\psi_3:(z,w)\mapsto\left(\frac{z+\sqrt{3}}{-\sqrt{3}z+1},\frac{16w}{(1-\sqrt{3}z)^4}\right).
\]

Let $(\hat{\alpha}_j,\hat{\beta}_j)$ ($j=1,2,3$), $(\hat{\omega}_1,\hat{\omega}_2,\hat{\omega}_3)$ and $\hat{A}$, $\hat{B}$ be as in p.~\pageref{paragraph:curve_and_cover_homology_bases}.
From the proof of Lemma \ref{lemm:double_cover_periods_auto2}, we know that
$\hat{A}_{11}=\hat{A}_{31}$, $\hat{A}_{13}=\hat{A}_{33}$, $\hat{A}_{12}=-\hat{A}_{32}$ and $\hat{A}_{22}=0$ on the one side, and $\hat{B}_{11}=\hat{B}_{31}$, $\hat{B}_{13}=\hat{B}_{33}$, $\hat{B}_{12}=-\hat{B}_{32}$ and $\hat{B}_{22}=0$ on the other side.

A direct calculation gives
\[
{\psi_3}^*(\hat{\omega}_1)=\frac14\hat{\omega}_1-\frac{\sqrt3}{2}\hat{\omega}_2+\frac34\hat{\omega}_3
\]
and, noting that ${\psi_3}_*(\hat{\alpha}_3)=\hat{\alpha}_1$, we get $\hat{A}_{12}=\frac{\sqrt3}{2}(\hat{A}_{11}-\hat{A}_{13})$.
Since ${\psi_3}_*(\hat{\beta}_1)=\hat{\beta}_3$, we also have $\hat{B}_{12}=\frac{\sqrt3}{2}(\hat{B}_{13}-\hat{B}_{11})$.

As  ${\psi_3}_*(\hat{\alpha}_1)=\hat{\alpha}_1-2\hat{\alpha}_2+\hat{\alpha}_3+\hat{\beta}_2$, we have
\begin{align*}
\int_{{\psi_3}_*(\hat{\alpha}_1)}\hat{\omega}_1 &= \hat{A}_{11}-2\hat{A}_{21}+\hat{A}_{31}+\hat{B}_{21}\\
=\int_{\hat{\alpha}_1}{\psi_3}^*(\hat{\omega}_1) &= \frac14\hat{A}_{11}-\frac{\sqrt3}{2}\hat{A}_{12}+\frac34\hat{A}_{13},
\end{align*}
hence $\hat{A}_{21}=\frac54\hat{A}_{11}-\frac34\hat{A}_{13}+\frac12\hat{B}_{21}$.
Noting that ${\psi_3}_*(\hat{\beta}_3)=\hat{\beta}_1+\hat{\beta}_2+\hat{\beta}_3$, the same arguments yield $\hat{B}_{21}=\frac32\hat{B}_{13}-\frac52\hat{B}_{11}$.

Considering $\int_{{\psi_3}_*(\hat{\alpha}_1)}\hat{\omega}_3=\int_{\hat{\alpha}_1}{\psi_3}^*(\hat{\omega}_3)$ with
\[
{\psi_3}^*(\hat{\omega}_3)=\frac34\hat{\omega}_1+\frac{\sqrt3}{2}\hat{\omega}_2+\frac14\hat{\omega}_3,
\]
we find $\hat{A}_{23}=\frac54\hat{A}_{13}-\frac34\hat{A}_{11}+\frac12\hat{B}_{23}$ and in a similar way, we finally obtain $\hat{B}_{23}=\frac32\hat{B}_{11}-\frac52\hat{B}_{13}$.

We have shown that $\hat{A}$ and $\hat{B}$ have the form
\begin{align*}
\hat{A} &=
\begin{pmatrix}
a_1 & \frac{\sqrt{3}}{2}(a_1-a_3) & a_3\\
\frac{5}{4}(a_1-b_1)+\frac{3}{4}(b_3-a_3) & 0 & \frac{3}{4}(b_1-a_1)+\frac{5}{4}(a_3-b_3)\\
a_1 & \frac{\sqrt{3}}{{2}}(a_3-a_1) & a_3
\end{pmatrix}\\
\hat{B} &= \begin{pmatrix}
b_1 & \frac{\sqrt{3}}{{2}}(b_3-b_1) & b_3\\
\frac{3}{2}b_3-\frac{5}{2}b_1 & 0 & \frac{3}{2}b_1-\frac{5}{2}b_3\\
b_1 & \frac{\sqrt{3}}{{2}}(b_1-b_3) & b_3
\end{pmatrix}
\end{align*}
with $a_1$, $a_3\in i\R$ and $b_1$, $b_3\in\R$ (by definition of the $\alpha_j$'s and $\beta_j$'s for $j=1,3$ and Remark \ref{rema:hyperelliptic_M-curve_homology_basis}).
Then the period matrix of $\hat{X}$ associated to the symplectic basis $\hat{\mathcal{B}}$ for $\HH_1(\hat{X},\Z)$ is
\begin{equation}
\hat{A}\hat{B}^{-1}=
\begin{pmatrix}
z_1 & \frac{1}{2}z_1 & z_{13}\\
\frac{1}{2}z_1 & \frac{1}{2}+\frac{3}{4}z_1-\frac{1}{2}z_{13} & \frac{1}{2}z_1\\
z_{13} & \frac{1}{2}z_1 & z_1
\end{pmatrix}
\end{equation}
where $z_1,z_{13}\in i\R$ are defined by
\begin{align*}
z_1 &= \frac{4}{3}\frac{a_3b_1-a_1b_3}{b_1^2-b_3^2}\\
z_{13} &= \frac{1}{3}\frac{a_3b_1-a_1b_3+3(a_1b_1-a_3b_3)}{b_1^2-b_3^2}
\end{align*}
hence the expression for $\hat{Z}$.
\end{proof}

We now compute the period matrix of the double cover of one particular real M-curve in the family $W_9$.
The natural choice consists in considering the curve defined by the translation surface tiled by three squares: by Proposition \ref{prop:autos_w9m} this is the only non-singular curve in this family admitting a non-hyperelliptic automorphism.
An equation of this curve was given in \S\ref{subsec:proof_autos_w9m}.

\begin{lemma}\label{lemm:3-square-tiled_periods}
Let $X_1$ be the real M-curve defined by the $3$-square-tiled translation surface, defined by the equation
\[
y^2=P(x)=x(x^2-1)(x-a^2)\left(x-\frac{1}{a^2}\right)\text{ with }a=2-\sqrt{3}
\]
and let $\hat{X}_1$ be its double cover.
Then the period matrix $\hat{Z}_1$ of $\hat{X}_1$ associated to $\hat{\mathcal{B}}$ is
\[
\hat{Z}_1=
\begin{pmatrix}
\frac43i&\frac23i&\frac13i\\
\frac23i&\frac12+\frac56i&\frac23i\\
\frac13i&\frac23i&\frac43i
\end{pmatrix}.
\]
\end{lemma}

\begin{proof}
The curve $X_1$ admits an order $4$ automorphism induced by $x\mapsto 1/x$.
Its double cover $\hat{X}_1$ is defined by the equation
\[
w^2=\hat{P}(z)=(z^4-1)(z^2-a^2)\left(z^2-\frac{1}{a^2}\right)
\]
and $\varphi$ lifts to an order $4$ automorphism of $\hat{X}_1$, denoted by $\psi_4$ and defined by
\[
\psi_4:(z,w)\mapsto\left(\frac{1}{z},\frac{iw}{z^4}\right).
\]
It is readily checked that
\[
{\psi_4}^*\left(\frac{dz}{w}\right)=i\frac{z^2dz}{w}\text{ and }{\psi_4}^*\left(\frac{z^2dz}{w}\right)=i\frac{dz}{w}.
\]
We also note that ${\psi_4}_*(\hat{\alpha}_1)=-\hat{\beta}_1$ hence, with the notations defined in the proof of Proposition \ref{prop:double_cover_periods_auto3}, we get $\hat{B}_{11}=-i\hat{A}_{13}$ and $\hat{B}_{13}=-i\hat{A}_{11}$.
A direct calculation then gives the period matrix $\hat{Z}_1$.
\end{proof}

\begin{remark}\label{rema:3-square-tiled_surface_period_matrix}
By Lemma \ref{lemm:double_cover_periods_auto2}, the period matrix of $X_1$ associated to the basis $\mathcal{B}$ is
\[
Z_1=
\begin{pmatrix}
1+\frac53i&\frac43i\\
\frac43i&\frac53i
\end{pmatrix}.
\]
\end{remark}

As the hyperelliptic curve $\hat{X}_1$ has many automorphisms, the modular transformation formula will enable us to identify the only even theta characteristic that vanishes for the period matrix $\hat{Z}_1$.

\begin{proposition}\label{prop:3-square-tiled_thetanull}
Let $X_1$ be the real M-curve defined by the $3$-square-tiled translation surface, let $\hat{Z}_1$ be the period matrix of its double cover $\hat{X_1}$ associated to $\hat{\mathcal{B}}$.
Then
\[
\thetasc{1&1&1\\ 1&0&1}(\hat{Z}_1)=0.
\]
\end{proposition}

\begin{proof}
Let $M=(\begin{smallmatrix}\alpha&\beta\\ \gamma&\delta\end{smallmatrix})\in\SP_6(\R)$ and
\[
p=\DIAG(\gamma\T{\delta})\text{ and }q=\DIAG(\alpha\T{\beta}).
\]
Then, following Theorem \ref{theo:theta_modular_formula}, we have the following transformation formula for theta characteristics:
\begin{equation} \label{eq:theta_modular_transformation}
\thetabc{2m\\2n}\big(M(Z)\big)=\phi(M,m,n,Z)\thetabc{2m'\\2n'}(Z),
\end{equation}
with
\begin{equation} \label{eq:characteristics_modular_transformation}
\begin{bmatrix}m'\\ n'\end{bmatrix}=\begin{bmatrix}
\T{\alpha}\left(m-\frac12p\right) +\T{\gamma}\left(n-\frac12q\right)\\
\T{\beta}\left(m-\frac12p\right) +\T{\delta}\left(n-\frac12q\right)
\end{bmatrix}=:M\scar{m\\ n}
\end{equation}
and where $\phi$ is a function, depending only on $M$, on $m,n\in(\rquotient{\mathbb{Z}}{2\mathbb{Z}})^3$ and on $Z\in\frak{S}_3$, that never vanishes.

By construction, $\hat{X}_1$ is a genus $3$ hyperelliptic curve, hence  by Theorem \ref{theo:hyperelliptic_genus3_thetanull}, there exists a unique even theta characteristic $\scar{m\\ n}=\scar{m_1&m_2&m_3\\ n_1&n_2&n_3}$ for which $\thetasc{m\\ n}(\hat{Z}_1)=0$.
Period matrices of the form \eqref{eq:genus3_auto3_periods} are stable under the action of the two symplectic matrices
\begin{align*}
M_2&=
\begin{pmatrix}
0&0&-1&0&0&0\\
0&-1&0&0&0&0\\
-1&0&0&0&0&0\\
0&0&0&0&0&-1\\
0&0&0&0&-1&0\\
0&0&0&-1&0&0
\end{pmatrix}
=\T{\rho_\Q(\psi_2)}\\
\text{and }
M_3&=\begin{pmatrix}
1&-2&1&0&1&0\\
1&-1&0&0&0&-1\\
1&0&0&0&0&0\\
0&0&0&0&0&1\\
0&0&0&0&-1&-2\\
0&0&0&1&1&1
\end{pmatrix}
=\T{\rho_\Q(\psi_3)},
\end{align*}
then by equation \eqref{eq:theta_modular_transformation} and Remark \ref{rema:theta-characteristics_mod2}, we must have $M_k\scar{m\\ n}\equiv\scar{m\\ n}\MOD2$ for $k=2$ and $3$.
As a consequence, we obtain
\[
\begin{cases}
m_1\equiv m_2\equiv m_3\MOD2\\
n_1\equiv m_2\equiv n_3\MOD2
\end{cases}
\]
which, among the even theta characteristics, reduces the possibilities to the following three:
\[
\thetasc{0&0&0\\0&1&0},\thetasc{1&1&1\\1&0&1}\text{ and }\thetasc{0&0&0\\0&0&0}
\]
the latter never being zero by Riemann theorem on theta divisor (for instance, see \cite[Theorem~VI.2.4]{farkas&kra92}).
Furthermore, the period matrix $\hat{Z}_1$ is stable under the action of
\[
M_4=
\begin{pmatrix}
0&0&0&1&0&0\\
0&1&0&0&-1&0\\
0&0&0&0&0&1\\
-1&0&0&0&0&0\\
0&2&0&0&-1&0\\
0&0&-1&0&0&0
\end{pmatrix}
=\T{\rho_\Q(\psi_4)},
\]
hence $M_4\scar{m\\ n}\equiv\scar{m\\ n}\MOD2$, which provides the extra condition
\[
m_2\equiv1\MOD2,
\]
hence the conclusion.
\end{proof}

Now we can prove the theorem:

\begin{proof}[Proof of Theorem \ref{theo:double_cover_thetanull}]
From now on we assume that the period matrices of the double covers are associated to the same symplectic basis $\hat{\mathcal{B}}$ for $\HH_1(\hat{X},\Z)$, such as these matrices have the form presented in the theorem.

The set of isomorphism classes of genus $2$ curves defined by an equation of the form
\begin{equation} \label{eq:genus2_order3_auto}
y^2=Q_s(x)=x(x+1)(x-s^2)\big(x-f_3(s)^2\big)\Big(x-f_3\big(f_3(s)\big)^2\Big).
\end{equation}
is connected, hence the associated family of curves also form a connected subset in the moduli space of genus $3$ hyperelliptic curves.
The map associating to the equation of such a double cover its period matrix is continuous, then the set of these matrices is connected.

By Theorem \ref{theo:hyperelliptic_genus3_thetanull}, exactly one even theta characteristic $\thetasc{m\\ n}$vanishes for these period matrices.
As the set of these period matrices is connected and the set of order $2$ theta characteristics is discrete, the same theta characteristic will vanish for the whole family.
Since by Proposition \ref{prop:3-square-tiled_thetanull}, the theta characteristic $\thetasc{1&1&1\\1&0&1}$ is zero for one member of this family, then it also vanishes for the period matrix of any double cover of a curve in the family defined by equation \eqref{eq:genus2_order3_auto}.
\end{proof}
%%%%%%%%%
%
\subsection{Periods of real M-curves from periods of double covers}
\label{subsec:period_of_curves_in_w9}
We now proceed to the proof of the main theorem.

\begin{proof}[Proof of Theorem \ref{theo:w9m_periods}]
Let $(X_1,\omega_1)$ be the L-shaped translation surface tiled by three squares; if $t$ is a real number such that $t>1$, then let $(X_t,\omega_t)=\left(\begin{smallmatrix}1&0\\0&t\end{smallmatrix}\right)\cdot(X_1,\omega_1)$.
By Remark \ref{rema:mcurves_are_teich_def}, Proposition \ref{prop:description_of_w9} and Lemma \ref{lemm:description_of_w9m},  every $(X_t,\omega_t)$ is defined by an equation of the form
\begin{align}
y^2&=Q_s(x)=x(x+1)(x-s^2)\Big(x-f_3\big(f_3(s)\big)^2\Big)\big(x-f_3(s)^2\big)\label{eq:3-square-tiled_surface_final}\\
\text{with }\omega_t&=\lambda_s\left(\frac{dx}{y}+\frac{xdx}{y}\right)\nonumber
\end{align}
for some real number $s$ such that $0<s<\sqrt3/3$ and some constant $\lambda_s\in\C^*$.

We construct from equation \eqref{eq:3-square-tiled_surface_final} a symplectic basis $\tilde{\mathcal{B}}=(\tilde{\alpha}_1,\tilde{\alpha}_2,\tilde{\beta}_1,\tilde{\beta}_2)$ for $\HH_1(X_t,\Z)$ as described in p.~\pageref{paragraph:hyperelliptic_homology_basis}: with the notations defined in Figure \ref{fig:3-square-tiled_deformation}, the point $P_1$ corresponds to $(-1,0)$, $P_2$ to $(0,0)$, $P_3$ to $\left(s^2,0\right)$, $P_4$ to $\left(f_3\big(f_3(s)\big)^2,0\right)$, $P_5$ to $\left(f_3(s)^2,0\right)$ and $P_6$ to the point at infinity.
For all $t\geq1$, we choose the constant $\lambda_s$ such that $\int_{\tilde{\beta}_1}\omega_t=1$.

\begin{figure}[!ht]
\centering \ifx\JPicScale\undefined\def\JPicScale{1}\fi
\unitlength \JPicScale mm
\begin{picture}(111.25,61.25)(0,0)
\linethickness{0.3mm}
\put(0,0){\line(0,1){40}}
\linethickness{0.3mm}
\put(0,40){\line(1,0){40}}
\linethickness{0.3mm}
\put(-0,0){\line(1,0){20}}
\linethickness{0.3mm}
\put(40,20){\line(0,1){20}}
\linethickness{0.3mm}
\put(20,0){\line(0,1){20}}
\linethickness{0.3mm}
\put(20,20){\line(1,0){20}}
\linethickness{0.3mm}
\put(30,30){\circle{1.88}}

\linethickness{0.3mm}
\put(30,40){\circle{1.88}}

\linethickness{0.3mm}
\put(30,20){\circle{1.88}}

\linethickness{0.3mm}
\put(10,30){\circle{1.88}}

\linethickness{0.3mm}
\put(0,10){\circle{1.88}}

\linethickness{0.3mm}
\put(20,0){\circle*{1.88}}

\linethickness{0.1mm}
\put(-0,20){\line(1,0){20}}
\linethickness{0.1mm}
\put(20,20){\line(0,1){20}}
\linethickness{0.3mm}
\put(20,10){\circle{1.88}}

\linethickness{0.3mm}
\put(0,0){\circle*{1.88}}

\linethickness{0.3mm}
\put(0,20){\circle*{1.88}}

\linethickness{0.3mm}
\put(0,40){\circle*{1.88}}

\linethickness{0.3mm}
\put(20,20){\circle*{1.88}}

\linethickness{0.3mm}
\put(40,20){\circle*{1.88}}

\linethickness{0.3mm}
\put(40,40){\circle*{1.88}}

\linethickness{0.3mm}
\put(20,40){\circle*{1.88}}

\linethickness{0.3mm}
\put(10,10){\circle{1.88}}

\put(20,50){\makebox(0,0)[cc]{$(X_1,\omega_1)$}}

\put(105,15){\makebox(0,0)[cc]{$(X_t,\omega_t)$}}

\put(16.25,23.75){\makebox(0,0)[cc]{$P_1$}}

\put(33.75,16.25){\makebox(0,0)[cc]{$P_2$}}

\put(33.75,33.75){\makebox(0,0)[cc]{$P_3$}}

\put(6.25,33.75){\makebox(0,0)[cc]{$P_4$}}

\put(6.25,6.25){\makebox(0,0)[cc]{$P_5$}}

\put(23.75,6.25){\makebox(0,0)[cc]{$P_6$}}

\linethickness{0.3mm}
\multiput(8.75,38.75)(0.12,0.12){21}{\line(1,0){0.12}}
\linethickness{0.3mm}
\multiput(-1.25,31.25)(0.12,-0.12){21}{\line(1,0){0.12}}
\linethickness{0.3mm}
\multiput(8.75,-1.25)(0.12,0.12){21}{\line(1,0){0.12}}
\linethickness{0.3mm}
\multiput(8.12,-1.25)(0.12,0.12){21}{\line(1,0){0.12}}
\linethickness{0.3mm}
\multiput(38.75,31.25)(0.12,-0.12){21}{\line(1,0){0.12}}
\linethickness{0.3mm}
\multiput(-1.25,11.25)(0.12,-0.12){21}{\line(1,0){0.12}}
\linethickness{0.3mm}
\multiput(18.75,11.25)(0.12,-0.12){21}{\line(1,0){0.12}}
\linethickness{0.3mm}
\multiput(28.75,38.75)(0.12,0.12){21}{\line(1,0){0.12}}
\linethickness{0.3mm}
\multiput(9.38,38.75)(0.12,0.12){21}{\line(1,0){0.12}}
\linethickness{0.3mm}
\multiput(28.75,18.75)(0.12,0.12){21}{\line(1,0){0.12}}
\linethickness{0.3mm}
\multiput(-1.25,31.88)(0.12,-0.12){21}{\line(1,0){0.12}}
\linethickness{0.3mm}
\multiput(38.75,30.62)(0.12,-0.12){21}{\line(1,0){0.12}}
\linethickness{0.1mm}
\multiput(10,30)(1.9,0){11}{\line(1,0){0.95}}
\linethickness{0.1mm}
\multiput(10,10)(0,1.9){11}{\line(0,1){0.95}}
\linethickness{0.1mm}
\multiput(30,20)(0,1.82){6}{\line(0,1){0.91}}
\linethickness{0.1mm}
\multiput(10,10)(1.82,0){6}{\line(1,0){0.91}}
\linethickness{0.3mm}
\multiput(78.75,58.75)(0.12,0.12){21}{\line(1,0){0.12}}
\linethickness{0.3mm}
\multiput(79.38,58.75)(0.12,0.12){21}{\line(1,0){0.12}}
\linethickness{0.3mm}
\multiput(98.75,28.75)(0.12,0.12){21}{\line(1,0){0.12}}
\linethickness{0.3mm}
\multiput(78.75,-1.25)(0.12,0.12){21}{\line(1,0){0.12}}
\linethickness{0.3mm}
\multiput(78.12,-1.25)(0.12,0.12){21}{\line(1,0){0.12}}
\linethickness{0.3mm}
\multiput(98.75,58.75)(0.12,0.12){21}{\line(1,0){0.12}}
\linethickness{0.3mm}
\multiput(68.75,46.25)(0.12,-0.12){21}{\line(1,0){0.12}}
\linethickness{0.3mm}
\multiput(108.75,46.25)(0.12,-0.12){21}{\line(1,0){0.12}}
\linethickness{0.3mm}
\multiput(68.75,16.25)(0.12,-0.12){21}{\line(1,0){0.12}}
\linethickness{0.3mm}
\multiput(88.75,16.25)(0.12,-0.12){21}{\line(1,0){0.12}}
\linethickness{0.3mm}
\multiput(108.75,45.62)(0.12,-0.12){21}{\line(1,0){0.12}}
\linethickness{0.3mm}
\multiput(68.75,46.88)(0.12,-0.12){21}{\line(1,0){0.12}}
\linethickness{0.3mm}
\put(70,0){\line(0,1){60}}
\linethickness{0.3mm}
\put(90,7.5){\line(0,1){22.39}}
\put(90,7.5){\vector(0,-1){0.12}}
\linethickness{0.3mm}
\put(70,60){\line(1,0){40}}
\linethickness{0.3mm}
\put(90,30){\line(1,0){13.75}}
\put(103.75,30){\vector(1,0){0.12}}
\linethickness{0.3mm}
\put(70,0){\line(1,0){20}}
\linethickness{0.3mm}
\put(110,30){\line(0,1){29.85}}
\linethickness{0.2mm}
\put(80,26.25){\line(0,1){33.58}}
\put(80,26.25){\vector(0,-1){0.12}}
\linethickness{0.2mm}
\put(92.5,45){\line(1,0){17.5}}
\put(92.5,45){\vector(-1,0){0.12}}
\linethickness{0.2mm}
\put(70,45){\line(1,0){22.5}}
\linethickness{0.3mm}
\put(90,0){\line(0,1){7.5}}
\linethickness{0.3mm}
\put(103.75,30){\line(1,0){6.25}}
\linethickness{0.2mm}
\put(80,0){\line(0,1){26.27}}
\put(102.5,26.25){\makebox(0,0)[cc]{$\tilde{\beta}_1$}}

\put(93.12,48.75){\makebox(0,0)[cc]{$\tilde{\beta}_2$}}

\put(83.75,26.27){\makebox(0,0)[cc]{$\tilde{\alpha}_2$}}

\linethickness{0.1mm}
\put(70,30){\line(1,0){20}}
\linethickness{0.1mm}
\put(90,30){\line(0,1){30}}
\put(93.75,7.61){\makebox(0,0)[cc]{$\tilde{\alpha}_1$}}

\linethickness{0.3mm}
\put(70,60){\circle*{1.88}}

\linethickness{0.3mm}
\put(110,60){\circle*{1.88}}

\linethickness{0.3mm}
\put(90,60){\circle*{1.88}}

\linethickness{0.3mm}
\put(90,30){\circle*{1.88}}

\linethickness{0.3mm}
\put(70,30){\circle*{1.88}}

\linethickness{0.3mm}
\put(110,30){\circle*{1.88}}

\linethickness{0.3mm}
\put(90,0){\circle*{1.88}}

\linethickness{0.3mm}
\put(70,0){\circle*{1.88}}

\linethickness{0.3mm}
\put(70,15){\circle{1.88}}

\linethickness{0.3mm}
\put(80,15){\circle{1.88}}

\linethickness{0.3mm}
\put(90,15){\circle{1.88}}

\linethickness{0.3mm}
\put(80,45){\circle{1.88}}

\linethickness{0.3mm}
\put(100,45){\circle{1.88}}

\linethickness{0.3mm}
\put(100,60){\circle{1.88}}

\linethickness{0.3mm}
\put(100,30){\circle{1.88}}

\end{picture}
\caption{Stretching the $3$-square-tiled surface} \label{fig:3-square-tiled_deformation}
\end{figure}
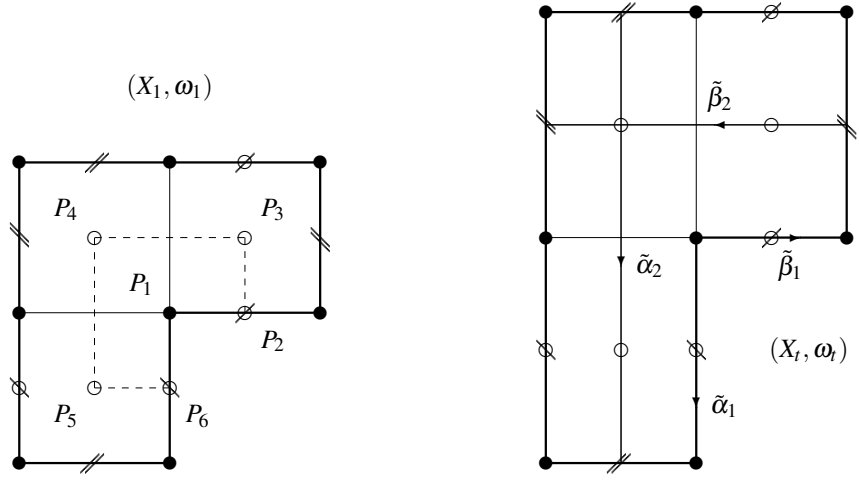

For $t$ being fixed, let $\eta_1=\omega_t$ and let $\eta_2$ be the unique holomorphic $1$-form on $X_t$ such that $\int_{\tilde{\beta}_1}\eta_2=0$ and $\int_{\tilde{\beta}_2}\eta_2=1$.
As $X_t$ is a real M-curve, and by construction of the homology basis $\tilde{\mathcal{B}}$, the period matrix $A_tB^{-1}$ must satisfy $\RE{(A_tB^{-1})}=0$.
A direct calculation then shows that there exist two unique real numbers $y_t$ and $y'_t$ such that
\[
A_t:=\left(\int_{\tilde{\alpha}_j}\eta_k\right)_{j,k}=\begin{pmatrix}-it & iy_t\\ -2it & iy'_t\end{pmatrix}\text{ and }B:=\left(\int_{\tilde{\beta}_j}\eta_k\right)_{j,k}=\begin{pmatrix}1&0\\-2&1\end{pmatrix},
\]
The basis $\mathcal{B}=(\alpha_1,\alpha_2,\beta_1,\beta_2)$ for $\HH_1(X_t,\Z)$ defined in p.~\pageref{paragraph:curve_and_cover_homology_bases} is obtained from $\tilde{\mathcal{B}}$ by the change of basis given by the symplectic matrix
\[
N=
\begin{pmatrix}
1&0&0&0\\
0&1&0&0\\
1&0&1&0\\
0&0&0&1
\end{pmatrix}
\in\SP_4(\Z).
\]
Then the period matrix of $X_t$ associated to $\mathcal{B}$ is
\[
Z_t=\T{N}(A_tB^{-1})=
\begin{pmatrix}
1+i(2y_t-t) & iy_t\\
iy_t & i(y_t/2+t)
\end{pmatrix}
\]
whose imaginary part is positive definite if, and only if $y_t>2t/3$.
Hence, following the notations of Theorem \ref{theo:double_cover_thetanull}, we have
\begin{align*}
z_1 &= iy_t\\
z_{13} &= i(t-y_t/2)
\end{align*}
from which we deduce the expression of the period matrix associated to the basis $\mathcal{B}$ of the double cover $\hat{X}_t$:
\[
\hat{Z}_t=\begin{pmatrix}iy_t&iy_t/2&i(t-y_t/2)\\ iy_t/2&\frac12+i(y_t-t/2)&iy_t/2\\ i(t-y_t/2)&iy_t/2&iy_t\end{pmatrix}.
\]
Consider the symplectic basis obtained from $\mathcal{B}$ by the change of basis given by
\[
M=
\begin{pmatrix}
0&0&1&0&0&0\\
-1&1&-1&0&0&0\\
1&0&0&0&0&0\\
1&1&1&0&1&1\\
1&0&1&0&1&0\\
1&1&1&1&1&0
\end{pmatrix}
\in\SP_6(\Z),
\]
then the period matrix associated to this new basis is
\[
\hat{Z}_t'=\T{M}(\hat{Z}_t)=
\begin{pmatrix}
\frac12+i\left(y_t-\frac12t\right)&\frac12-\frac12i(y_t-t)&\frac12-\frac12i(y_t-t)\\
\frac12-\frac12i(y_t-t)&\frac12+i\left(y_t-\frac12t\right)&\frac12-\frac12i(y_t-t)\\
 \frac12-\frac12i(y_t-t)&\frac12-\frac12i(y_t-t)&\frac12+i\left(y_t-\frac12t\right)
 \end{pmatrix},
 \]
fo which we can prove, by the same arguments as those used in the proof of Theorem \ref{theo:double_cover_thetanull}, that the corresponding vanishing even theta characteristic is $\thetasc{1&1&1\\0&0&0}$.
Developping, we get
\begin{multline*}
\thetasc{1&1&1\\0&0&0}(\hat{Z}_t')=
\sum_{(k_1,k_2,k_3)\in\Z^3}\exp\pi\Biggl[y_t\Big(k_1k_2+k_2k_3+k_3k_1-(k_1^2+k_2^2+k_3^2)\Big)\\
+t\left(\frac12(k_1^2+k_2^2+k_3^2)-(k_1k_2+k_2k_3+k_3k_1)-\frac12(k_1+k_2+k_3)-\frac38\right)\\
+i\left(\frac12(k_1^2+k_2^2+k_3^2)+k_1k_2+k_2k_3+k_3k_1+\frac32(k_1+k_2+k_3)+\frac98\right)\Biggr],
\end{multline*}
hence equation \eqref{eq:main_equation}.
\end{proof}
%%%%%%%%%%%%%%%%%%%%%%%%%%%%%%
%
%%%%%%%%%%%%%%%%%%%%%%%%%%%%%%
%
\bibliographystyle{amsalpha}
\bibliography{ref}
\end{document}